\documentclass{amsart}
\usepackage{etex}
\usepackage[T1]{fontenc}
\usepackage[utf8]{inputenx}
\usepackage{lmodern}
\usepackage[pdftex,pagebackref]{hyperref}
 \hypersetup{
   colorlinks   = true,  
   urlcolor     = black, 
   linkcolor    = black, 
   citecolor    = black  
 }
\usepackage[kerning=true,tracking=true]{microtype}

\title{Holomorphic Cartan geometries and rational curves}

\author{Indranil Biswas}
\address{Tata Institute of Fundamental Research, Homi Bhabha Road, 
Mumbai 400005, India}
\email{indranil@math.tifr.res.in}

\author{Benjamin McKay}
\address{University College Cork, Cork, Ireland}
\email{b.mckay@ucc.ie}

\date{\today} 


\usepackage{tikz}
\usepackage{pdftexcmds}
\usepackage{ifthen}
\usepackage{ifpdf}
\usepackage{amssymb}
\usepackage{amsfonts}
\usepackage{euscript}
\usepackage{braket}
\usepackage{tikz-cd}
\usepackage{array}
\usepackage{booktabs}
\usepackage{varioref}
\newcolumntype{L}{>{$}l<{$}}
\newtheorem{theorem}{Theorem}
\newtheorem{corollary}{Corollary}
\newtheorem{lemma}{Lemma}
\newtheorem{proposition}{Proposition}
\theoremstyle{remark}

\newtheorem{definition}{Definition}

\newcommand{\LieDer}{\ensuremath{\EuScript L}}

\newcommand{\hook}{\ensuremath{\mathbin{ \hbox{\vrule height1.4pt
 width4pt depth-1pt \vrule height4pt width0.4pt depth-1pt}}}}

\newcommand{\C}[1]{\ensuremath{\mathbb{C}^{#1}}}

\newcommand{\Proj}[1]{\mathbb{P}^{#1}}
\newcommand{\OO}[1]{
 \ensuremath{
    \mathcal{O}
    \ifthenelse{\equal{#1}{0}}
      {}
      {\of{#1}}
  }
}
\newcommand{\OOp}[2]{
  \ensuremath{
    \mathcal{O}
    \ifthenelse{\equal{#1}{0}}
      {}
      {\left({#1}\right)}
    \ifthenelse{\equal{#2}{1}}
      {}
      {^{\oplus{#2}}}
  }
}
\DeclareMathOperator{\Ad}{Ad}
\newcommand*{\Lm}[2]{\ensuremath{\Lambda^{#1}\!\left({#2}\right)}}
\newcommand*{\nForms}[2]{\ensuremath{\Omega^{#1}\!\left({#2}\right)}}
\newcommand*{\cohomology}[2]{\ensuremath{H^{#1}\of{#2}}}

\newcommand{\LieG}{\ensuremath{\mathfrak{g}}}
\newcommand{\LieH}{\ensuremath{\mathfrak{h}}}

\newcommand{\LieP}{\ensuremath{\mathfrak{p}}}

\newcommand{\RE}[1]{\ensuremath{R_{#1}}}

\newcommand*{\pr}[1]{\ensuremath{\left(#1\right)}}
\makeatletter
\newcommand*{\normalBundle}[3][.]%
{
\ensuremath{N\ifnum\pdf@strcmp{#1}{.}=\z@_{#3/#2}\else_{#1,#3/#2}\fi}
}
\makeatother

\newcommand*{\defeq}{\mathrel{\vcenter{\baselineskip0.5ex \lineskiplimit0pt
                     \hbox{\scriptsize.}\hbox{\scriptsize.}}}%
                     =}
\newcommand*{\of}[1]{\ensuremath{\!\left(#1\right)}}

\newcommand{\amal}[3]{\ensuremath{{#1} \mathbin{\times^{#2}} \! #3}}

\newcounter{remarkCounter}
\newcommand*{\DashedArrow}[1][]{\mathbin{\tikz [baseline=-0.25ex,-latex, dashed,->,densely dashed,#1] \draw [#1] (0pt,0.5ex) -- (1.5em,0.5ex);}}%

\begin{document}
\begin{abstract}
We prove that any compact K\"ahler manifold bearing a holomorphic Cartan geometry contains a rational curve just when the Cartan geometry is inherited from a holomorphic Cartan geometry on a lower dimensional compact K\"ahler manifold. 
This shows that many complex manifolds admit no or few holomorphic Cartan geometries.
\end{abstract}
\maketitle
{\hspace{.8cm}
\begin{minipage}{.7\textwidth}
\tableofcontents
\end{minipage}}

\section{Introduction}

\subsection{Convexity}

A \emph{rational curve} in a complex manifold \(M\) is a non-constant holomorphic map \(f \colon \Proj{1} \to M\); it is \emph{free} if \(f^* TM\) is spanned by its global sections, \emph{ample} if \(f^* TM\) is a sum of line bundles of positive degree.
A complex manifold is \emph{convex} if every rational curve in it is free. 
We will define Cartan geometries in section~\vref{section:Definitions}.
On page~\pageref{proof:thm:FreeRats} we will prove:

\begin{theorem}%
\label{thm:FreeRats}
Any complex manifold \(M\) bearing a holomorphic Cartan geometry is convex.
\end{theorem}

A proper surjective holomorphic map \(Z \to Y\) between complex spaces is a \emph{modification} if it is a biholomorphism away from some closed complex subvarieties in \(Y\) and \(Z\) of positive codimension.
On page~\pageref{proof:cor:BlowsDown} we will prove:

\begin{corollary}%
\label{cor:BlowsDown}
Suppose that \(f \colon Z \to Y\) is a modification of reduced complex spaces with \(Y\) smooth.
Suppose that there is a smooth point of \(Z\) near which \(f\) is not a local isomorphism.
Then the set of smooth points of \(Z\) bears no holomorphic Cartan geometry.
\end{corollary}
For example, if \(f \colon Z \to Y\) is the blowup of a subvariety \(Y_0 \subset Y\), then over smooth points of \(Y_0\), \(Z\) is smooth and \(f\) is not a local isomorphism, so \(Z\) admits no holomorphic Cartan geometry.
Proposition 2.1 of \cite{McKay:2013} is a similar result.

\subsection{The main theorem}

Any complex space \(X\) has a complex space \(D=D_X\), the \emph{Douady space} of \(X\), and subspace \(Z \subset D \times X\), the \emph{universal flat family}, uniquely determined up to isomorphism so that 
\begin{enumerate}
\item \(Z \to D\) is flat and proper and
\item if \(D'\) is a complex space and \(Z' \subset D' \times X\) is flat and proper over \(D'\) then there exists a unique holomorphic map \(D' \to D\) such that
\(Z' = D' \times_D Z\).
\end{enumerate}
For proof, see \cite[p. 50, Cor. 6.0.3]{Magnusson:2013}.
A rational curve is \emph{tame} if its graph lies in a compact irreducible component of the Douady space of compact complex subspaces of \(\Proj{1} \times M\).
A complex manifold is \emph{tame} if every rational curve in it is tame. 
Compact complex curves and surfaces are tame \cite[p. 387]{Gauduchon:1977}. 
Holomorphic bundles of tame complex spaces with tame base have tame total space.

A \emph{Fujiki} space is a compact complex space whose reduction is the meromorphic (or equivalently, holomorphic) image of a compact K\"ahler manifold \cite{AKMW:2002} \S{}1.2.4.
For any Fujiki space \(S\), each component of the Douady space of any reduced compact complex subspace of \(S\) is a Fujiki space \cite[p. 189]{Fujiki:1982}, hence compact.
If \(S\) is Fujiki then \(\Proj{1} \times S\) is also Fujiki, so Fujiki manifolds are tame.
A reduced and irreducible compact complex space is \emph{Moishezon} if it is bimeromorphic to a smooth complex projective variety \cite[p. 26, Thm 3.6]{Ueno:1975}; Moishezon spaces are Fujiki and so tame.
A \emph{rational homogeneous variety} is a connected and simply connected complex projective variety with a holomorphic, transitive and effective action of a connected complex Lie group.

We will define \emph{dropping} of Cartan geometries on \vpageref{page:drop}.
On page~\pageref{proof:thm:OneCurveForcesDescent} we will prove our main theorem:
\begin{theorem}%
\label{thm:OneCurveForcesDescent}
Take a complex homogeneous space \((X,G)\), \(X=G/H\).
Let \(M\) be a connected tame compact complex manifold with a holomorphic \((X,G)\)-geometry. 
Then there is a holomorphic fiber bundle map \(M \to M'\), the maximal rationally connected fibration \cite[p. 222]{Kollar:1996} of \(M\), such that
\begin{enumerate}
\item\label{item:no.rat}
\(M'\) contains no rational curves and
\item 
every fiber of \(M \to M'\) is biholomorphic to a fixed rational homogeneous variety \(H'/H\) for a unique subgroup \(H \subset H' \subset G\) and
\item 
every holomorphic Cartan geometry on \(M\) pulls back to each fiber of \(M \to M'\) to a holomorphic Cartan geometry equivalent to the model Cartan geometry on \(H'/H\) and
\item
every holomorphic Cartan geometry on \(M\) drops to a holomorphic Cartan geometry on \(M'\) and
\item\label{item:FKP}
if \(M\) is Fujiki, K\"ahler, projective then \(M'\) is Fujiki, K\"ahler, projective and
\item
if the Cartan geometry on \(M\) drops via a holomorphic map \(M \to M''\) then this map factors through unique drops \(M \to M' \to M''\).
\end{enumerate}
\end{theorem}

Clearly \(M'\) is a point just when \(M\) is a rational homogeneous variety with its model geometry.
For example, a connected tame compact complex manifold containing a rational curve and bearing a holomorphic projective or conformal connection is biholomorphic to projective space or a smooth quadric hypersurface with its standard flat projective or conformal Cartan connection.

On page~\pageref{proof:corollary:Moishezon} we will prove:
\begin{corollary}%
\label{corollary:Moishezon}
Any compact Moishezon manifold which bears a holomorphic Cartan geometry is a smooth complex projective variety.
\end{corollary}

On page~\pageref{proof:corollary:ample} we will prove:
\begin{corollary}\label{corollary:ample}
Suppose that \(M\) is a compact connected complex manifold.
Then the following are equivalent:
\begin{enumerate}
\item
The manifold \(M\) bears an ample rational curve and a holomorphic Cartan geometry.
\item
The manifold \(M\) is a rational homogeneous variety. 
Every holomorphic Cartan geometry on \(M\) is the model geometry of \(M\) under the action of some transitive complex subgroup of the biholomorphism group of \(M\), up to isomorphism of complex homogeneous spaces.
\end{enumerate}
\end{corollary}

\subsection{Symmetries}

\begin{proposition}
Suppose that \(M\) is a compact connected Moishezon manifold with a holomorphic Cartan geometry.
Either 
\begin{enumerate}
  \item the geometry drops as in Theorem~\vref{thm:OneCurveForcesDescent} or
  \item \(M\) is an abelian variety, up to finite unramified covering or
  \item the biholomorphism group of \(M\) is finite. 
\end{enumerate}
\end{proposition}
\begin{proof}
Note that \(M\) is a smooth projective variety by Corollary~\vref{corollary:Moishezon}.
If a holomorphic vector field on \(M\) has a zero, and is not everywhere zero, then \(M\) contains a rational curve \cite[p. 319, Cor. 3.5]{Jahnke/Peternell/Radloff:2006}, so the Cartan geometry drops by Theorem~\ref{thm:OneCurveForcesDescent}.
If there is a holomorphic vector field on \(M\) with no zeroes, then the Bott residue formula ensures that \(M\) has a finite unramified covering by an abelian variety.
\end{proof}

On the other hand, suppose that \(M\) is a compact connected complex manifold with a holomorphic Cartan geometry, but \emph{not} a smooth projective variety and so not a Moishezon manifold.
A complex homogeneous space \((X,G)\) is algebraic if \(X\) is a complex algebraic variety with an effective algebraic group action of a complex algebraic group \(G\).
Dumitrescu \cite{Dumitrescu:2011} proved that near the generic point of \(M\) there is a nonzero holomorphic vector field whose flow preserves every holomorphic Cartan geometry on \(M\) with algebraic model and also preserves every holomorphic tensor on \(M\).
In particular, every such Cartan geometry has local symmetries.
On page~\pageref{proof:theorem:bimero} we will prove:
\begin{theorem}\label{theorem:bimero}
Suppose that \(M_0\) and \(M_1\) are tame compact complex manifolds bearing holomorphic Cartan geometries with the same model.
Every bimeromorphism \(M_0 \DashedArrow[<->] M_1\) which identifies the Cartan geometries on a Zariski dense open set extends to a unique biholomorphism.
\end{theorem}

\section{Definitions}\label{section:Definitions}

\subsection{Cartan geometries}

Throughout we use the convention that structure groups of principal bundles act on the right. 
If \(E \to M\) is a principal \(G\)-bundle, denote the \(G\)-action as \(r_g e = eg\), where \(e \in E\) and \(g\in G\). 

\begin{definition}\label{def:CartanConnection}
Take a Lie group \(H\) with Lie algebra \(\LieH\) and an \(H\)-module \(V\) containing \(\LieH\) as an \(H\)-submodule.
A \(V,H\)-geometry on a manifold \(M\) is a choice of \(C^\infty\) principal \(H\)-bundle \(E \to M\) with a smooth 1-form \(\omega \in \nForms{1}{E} \otimes V\) called the \emph{Cartan connection}, so that:
\begin{enumerate}
\item
\(
r_h^* \omega = \Ad_h^{-1} \omega
\) for all \(h \in H\).
\item\label{item:isom}
\(\omega_e \colon T_e E \to V\) is a linear isomorphism at each point
\(e \in E\).
\item
For each \(A \in V\), define a vector field \(\vec{A}\) on \(E\) by
the equation \(\vec{A} \hook \omega = A\); this defines \(\vec{A}\) uniquely by (\ref{item:isom}). 
For every \(A \in \LieH\), the vector field \(\vec{A}\) coincides with the one arising from the right principal bundle action of \(H\).
\end{enumerate}
Take a manifold \(X\) homogeneous for the action of a Lie group \(G\), with a chosen point \(x_0 \in X\) having stabilizer \(H \subset G\).
An \((X,G)\)-geometry, also called a \emph{Cartan geometry} modelled on \((X,G)\), is a \(V,H\)-geometry for \(V\defeq\LieG\) the Lie algebra of \(G\), which is an \(H\)-module by restriction of the adjoint \(G\)-action to \(H\). 
\end{definition}

For any homogeneous space \((X,G)\) with chosen point \(x_0\), the principal bundle \(g \in G \mapsto gx_0 \in X\) is an \((X,G)\)-geometry, called the \emph{model Cartan geometry}, with Cartan connection \(\omega=g^{-1} \, dg\), the left invariant Maurer--Cartan 1-form on \(G\).

Suppose that \(H \subset H'\) is a closed subgroup, and \(E \to M'\) is a \(V,H'\)-geometry.
Then \(E \to M \defeq E/H\) is a \(V,H\)-geometry with the same Cartan connection as the original \(V,H'\)-geometry.
Moreover \(M \to M'\) is a fiber bundle with fiber \(H'/H\).
The geometry \(E \to M\) is called the \emph{\(V,H\)-lift} of \(E \to M'\) (or simply the \emph{lift}). 
Conversely, a given \(V,H\)-geometry \emph{drops}\label{page:drop} to a certain \(V,H'\)-geometry if it is isomorphic to the \(V,H\)-lift of the \(V,H'\)-geometry in question.
For example, if \(X\) is connected homogeneous \(G\)-space then an \((X,G)\)-geometry on a connected manifold drops to a \((*,G)\)-geometry (where \(*\) is a point with trivial \(G\)-action) just when the geometry is isomorphic to its model.

A \emph{complex homogeneous space} is a complex manifold \(X\) with transitive holomorphic action of a complex Lie group \(G\).
A Cartan geometry modelled on a complex homogeneous space is \emph{holomorphic} if the bundle is a holomorphic principal bundle and the Cartan connection is a holomorphic 1-form.
Isomorphisms of holomorphic Cartan geometries are biholomorphic.

\subsection{Connections and Cartan connections}

Let \(\omega\) be the Cartan connection of a holomorphic Cartan geometry \(E \to M\) modelled on a homogeneous space \(X=G/H\).
The Lie algebra \(\LieG\) is equipped with the adjoint action of \(H\), so \(\LieG/\LieH\) is also an \(H\)-module.

\begin{lemma}%
[{\cite[p. 188, Theorem 3.15]{Sharpe:1997}}]%
\label{lemma:TgtBundle}
The Cartan connection of any Cartan geometry \(\pi \colon E \to M\) determines isomorphisms for any points \(m \in M\) and \(e \in E_m\)
\[
\begin{tikzcd}
0 \arrow{r} & \ker \pi'(e) \arrow{r} \arrow{d}{\sim} & T_e E \arrow{r} \arrow{d}{\sim} & T_m M  \arrow{r} \arrow{d}{\sim} & 0 \\ 
0 \arrow{r} & \LieH \arrow{r} & \LieG \arrow{r} & \LieG/\LieH  \arrow{r} & 0 
\end{tikzcd}
\]
Thus \(TM\) is identified with the vector bundle \(\amal{E}{H}{\of{\LieG/\LieH}}\)
associated to the principal \(H\)-bundle \(E\) for the \(H\)-module \(\LieG/\LieH\).
\end{lemma}

Note that \(E\times \pr{\LieG/\LieH}\) is a quotient bundle of \(TE\,=\, E\times \LieG\). 
Consider the \(\LieG\)--valued 2--form
\begin{equation}\label{2-form-i-a}
\nabla \omega = d \omega + \frac{1}{2}\left[\omega,\omega\right]
\end{equation}
on \(E\).
Lemma \ref{lemma:TgtBundle} identifies \(\nabla \omega\) with an \(H\)-equivariant function \(K \colon E \to \LieG \otimes \Lm{2}{\LieG/\LieH}^*\), called the \emph{curvature} of the Cartan connection. 
A Cartan geometry is called \emph{flat} if \(K=0\).

Given any \(H\)-bundle \(E \to M\), and an inclusion of Lie groups \(H \subset G\), we define on \(E \times G\) the \(G\)-action \((e,g)g'=\pr{e,gg'}\) and the \(H\)-action \((e,g)h=\pr{eh,h^{-1}g}\).
Let \(E_G \defeq \amal{E}{H}{G}\) be the quotient by the \(H\)-action, a principal right \(G\)-bundle over \(M\).
Given a Cartan connection \(\omega\) on \(E\), we let 
\[
\omega_G \defeq \Ad_g^{-1} \omega + g^{-1} \, dg,
\]
and check that \(\omega_G\) is \(H\)-invariant and \(G\)-equivariant, and is the pullback of a unique connection, also denoted \(\omega_G\), on \(E_G\).
Let \(\amal{E}{H}{\LieG}\) be the vector bundle over \(M\) associated to
the principal \(H\)-bundle \(E\) for the \(H\)-module \(\LieG\).
The curvature \(K\) descends 
to a holomorphic section of \(\pr{\amal{E}{H}{\LieG}} \otimes \nForms{2}{M}\) on \(M\), which we also denote by \(\nabla \omega\). 
The section \(\nabla \omega\) is the curvature of the connection \(\omega_G\) on \(E_G\).

\subsection{Development}\label{subsection:Development}

We will define the well known concept of developing map \cite[p. 3]{Goldman:2010}, \cite[p. 368]{Sharpe:1997}.
Take a complex homogeneous space \((X,G)\) and a holomorphic Cartan geometry \(H \to E_H \to M\) with Cartan connection \(\omega\).
An \emph{integral} of the Cartan geometry is a morphism \(f \colon Z \to M\) of complex spaces so that the holomorphic connection \(\omega_G\) on \(E_G \defeq \amal{E_H}{H}{G}\) pulls back to a flat connection on \(f^* E_G \to Z\).

Suppose that \(p \colon Z \to Y\) is a holomorphic map of complex analytic spaces.
Take a holomorphic map \(f \colon Z \to M\)  to a complex manifold \(M\). 
Let \(Z_y \defeq p^{-1}(y) \subset Z\) and let \(f_y=\left.f\right|_{Z_y}\).
If \(Z_y\) is simply connected and \(f_y \colon Z_y \to M\) is an integral for every \(y \in Y\), say that 
\[
\begin{tikzcd}
Z \arrow{r}{f} \arrow{d}{p} & M \\
Y
\end{tikzcd}
\]
is a \emph{family of integrals}.
Since every fiber \(Z_y\) is simply connected, each bundle \(f_y^*E_G \to Z_y\) is \(G\)-equivariantly trivial.
Choice of one point of \(f_y^* E_G\) picks out a particular trivialization over the fiber \(Z_y\).

Suppose that \(s \colon Y \to f^* E_G\) is a holomorphic section of the composition \(f^*E_G \to Z \to Y\).
Call the data
\[
\begin{tikzcd}
f^*E_H \arrow{r} \arrow{d} & E_H \arrow{d}{\pi} & H \arrow{l} \\
Z \arrow{r}{f} \arrow{d}{p} & M \\
Y \arrow[bend left]{uu}{s}
\end{tikzcd}
\]
a \emph{framed family of integrals} with \emph{frame} \(s\). 
A \emph{bundle developing map} on a framed family of integrals  is a holomorphic  \(H\)-equivariant map \(f^*E_H \to G\) so that 
\begin{enumerate}
\item the bundle developing map takes the image of the frame \(s\) to \(1\) and 
\item above each fiber \(Z_y\) the graph of the bundle developing map is a leaf of the foliation \(dg \, g^{-1} = -\omega\).
\end{enumerate}
Quotienting the bundle developing map by \(H\)-action gives the \emph{developing map} \(f_1 \colon Z \to X\) of the framed family of integrals.
\begin{lemma}
Every framed family of integrals has a unique bundle developing map.
\end{lemma}
\begin{proof}
The choice of point \(s(y) \in f_y^* E_G\) picks out a particular trivialization of the bundle \(f^* E_G \to Z\); identify \(Z\) with its image in \(f^*E_G\) by that trivialization.
The principal \(H\)-bundle \(f^*E_H \times G \to f^*E_G\) pulls back to that copy of \(Z \subset f^*E_G\) to a principal \(H\)-bundle \(B \to Z\).
The projection \(f^*E_H \times G \to f^*E_H\) onto the first factor restricts to \(B\) to an isomorphism \(B=f^*E_H\).
So \(B\) sits in \(f^*E_H \times G\) as just the graph of an \(H\)-equivariant map \(f^*E_H \to G\), our candidate for the bundle developing map.
Because we constructed our bundle \(B\) above a parallel section \(Z \subset f^*E_G\), we have \(\omega_G=0\) on \(B\), i.e. the bundle developing map satisfies \(0=\Ad_g^{-1} \omega + g^{-1} \,dg\), which we can write as \(dg \, g^{-1} = -\omega\).
Retracing our steps, we uniquely determine the bundle developing map of any framed family of integrals by the requirements given above.
\end{proof}

\begin{lemma}\label{remark:TMZeroEqualsTMOne}
Take a complex manifold \(M\) with Cartan geometry \(E \to M\) modelled on a complex homogeneous space \((X,G)\), and a framed family \(f\) of integrals with developing map \(f_1\).
The following pullback bundles have canonical isomorphisms as indicated:
\begin{align*}
f^* E &\cong f_1^* G \text{ as holomorphic principal \(H\)-bundles}, \\
     f^* TM &\cong f_1^* TX  \text{ as holomorphic vector bundles}, \\
     f^* \pr{\amal{E}{H}{\LieG}} &\cong f_1^* \pr{\amal{G}{H}{\LieG}} \text{as holomorphic vector bundles with connection}.
\end{align*}
\end{lemma}
\begin{proof}
Let \(E_H \defeq E\).
Note that \(E_H \times^H \LieG = E_G \times^G \LieG\) is a holomorphic vector bundle with holomorphic connection \(\omega_G\).
By lemma~\vref{lemma:TgtBundle}, it suffices to identify \(f^* E_H=f_1^*G\) as principal \(H\)-bundles.
The bundle developing map is precisely the required isomorphism.
\end{proof}

\section{Development of rational curves}%
\label{section:DevelopmentOfCurves}
Note that any family of simply connected complex curves in a complex manifold with Cartan geometry is a family of integrals, because the curvature of the holomorphic connection \(\omega_G\) is a holomorphic 2-form valued in a vector bundle, and so vanishes on curves.

We now prove Theorem \ref{thm:FreeRats}.

\begin{proof}%
\label{proof:thm:FreeRats}
Suppose that \(E \to M\) is a holomorphic Cartan geometry, say with model \((X,G)\).
The developing map \(f_1 \colon \Proj{1} \to X\) has associated isomorphism \(f^*TM \to f_1^* TX\).
The Lie algebra \(\mathfrak g\) of \(G\) acts on \(X\) as a collection of holomorphic
vector fields. 
Since the action of \(G\) on \(X\) is transitive, the tangent space of \(X\) at each point is spanned by the image of \(\mathfrak g\).
In particular, \(TX\) is spanned by its global holomorphic sections.
Pulling back those sections via \(f_1\), the vector bundle \(f_1^* TX\) is also spanned by its global holomorphic sections, i.e. \(f_1\) is free.
By Lemma \ref{remark:TMZeroEqualsTMOne}, the vector bundle \(f^* TM = f_1^* TX\) is also spanned by its global holomorphic sections, so \(f\) is also free.
\end{proof}

\begin{lemma}%
\label{lemma:adGTrivial}
Suppose that \((X,G)\) is a complex homogeneous space \(X=G/H\) and that \(E \to M\) is an \((X,G)\)-geometry, and \(f \colon \Proj{1} \to M\) is a rational curve. 
Then the holomorphic vector bundle \(f^*\of{\amal{E}{H}{\LieG}} \,=\, \amal{(f^*E)}{H}{\LieG}\)
on \(\Proj{1}\) is trivial. 
\end{lemma}

\begin{proof}
Every holomorphic vector bundle on \(\Proj{1}\) splits into a direct sum of holomorphic line bundles \cite[p. 122, Th\'eor\`eme 1.1]{Grothendieck:1957}.
Since \(f^*\of{\amal{E}{H}{\LieG}}\) admits a holomorphic connection, all direct summands of \(f^*\of{\amal{E}{H}{\LieG}}\) have degree zero \cite[p. 201, Thm. 8]{Atiyah:1957}. 
Therefore, the holomorphic vector bundle \(f^*\of{\amal{E}{H}{\LieG}}\) is trivial.
\end{proof}

\begin{lemma}%
\label{lemma:UnobstructedDeformation}
Suppose that \(M\) is a complex manifold bearing a holomorphic Cartan geometry. 
Then every rational curve in \(M\) is a smooth point of the space of all of its deformations.
\end{lemma}

\begin{proof}
Since every rational curve is free, we can apply Horikawa's deformation theorem \cite[p. 649]{Horikawa:1974}.
\end{proof}

The following lemma is essentially Theorem 2 of \cite[p. 55]{Hwang/Mok:1997}.

\begin{lemma}%
\label{lemma:TgtKillsCurvature}
Suppose that \(M\) is a complex manifold bearing a holomorphic Cartan geometry and \(f \colon \Proj{1} \to M\) is a rational curve. 
Fix a decomposition \(f^* TM = \bigoplus_{d \ge 0} \OOp{d}{p_d}\).
Let \[T^+_f \defeq \bigoplus_{d > 0} \OOp{d}{p_d} \subset f^* TM\] be the sum of all
positive degree line subbundles of \(f^* TM\). 
Then \(T^+_f \subset f^* TM\) is a  holomorphic vector bundle, containing \(f_* T\Proj{1}\) (and so
of rank at least one). 
Moreover, if \(v \in T^+_f\), then \(v \hook f^* \nabla \omega=0\), where \(\nabla \omega\) as before denotes the curvature of any Cartan geometry on \(M\).
\end{lemma}

\begin{proof}
By Theorem \ref{thm:FreeRats}, \(f\) is free. 
Therefore, \(f^*TM\) is a direct sum of holomorphic line bundles of nonnegative degree, and \(T^+\) is a vector subbundle. 
Since \(f\) is not constant, \(f' \colon T\Proj{1} = \OO{2} \to f^* TM\) is not a zero homomorphism, so \(f^* TM\) contains a line subbundle of degree at least 2. 
Since the quotient \((f^*TM)/T^+_f\) is a trivial vector bundle, there is no nonzero homomorphism from \(\OO{2}\) to \((f^*TM)/T^+_f\). 
Consequently, \(f'(T\Proj{1})\) is contained in \(T^+_f\).

By Lemma \ref{lemma:adGTrivial}, the pullback \(f^* \nabla \omega\) is a section of \(\OOp{0}{\dim G} \otimes {f^* T^*M}\otimes {f^* T^*M}\). 
Any vector \(v \in T^+_f\) extends to a holomorphic section \(\tilde{v}\) of \( T^+_f\), and \(\tilde{v} \hook f^* \nabla \omega\) is a holomorphic section of \(\OOp{0}{\dim G} \otimes {f^* T^*M}\), which is a direct sum
of holomorphic line bundles of nonpositive degree. 
But \(\tilde{v}\) vanishes at some point, and hence the section \(\tilde{v} \hook f^* \nabla \omega\) vanishes at that point.
Being a section of a line bundle of nonpositive degree, \(\tilde{v} \hook f^* \nabla \omega\) vanishes identically because it vanishes at some point.
\end{proof}

\section{The local consequences of rational curves}%
\label{section:Local}

\begin{theorem}%
\label{theorem:local}
Suppose that \((X,G)\) is a complex homogeneous space, \(M\) is a connected complex manifold containing a rational curve, and \(E \to M\) is a holomorphic \((X,G)\)-geometry.
Then there is a complex Lie subgroup \(H' \subset G\) strictly containing \(H\), and a covering of \(M\) by  open subsets \(M_{\alpha}\) so that:
\begin{enumerate}
\item 
The restriction of the \((X,G)\)-geometry to \(M_{\alpha}\) is locally isomorphic to the lift of a holomorphic \(\LieG,H'\)-geometry.
\item
The subbundle \(\amal{E}{H}{\pr{\LieH'/\LieH}} \subset \amal{E}{H}{\pr{\LieG/\LieH}}=TM\) is the tangent bundle of a holomorphic foliation.
\item
The \((X,G)\)-geometry restricts to a flat holomorphic \((X',H')\)-geometry (with \(X'=H'/H\)) on each leaf of the foliation.
\item
Every rational curve in \(M\) lies inside one of the leaves of the foliation.
\end{enumerate}
\end{theorem}

\begin{proof}
Denote the bundle map \(E \to M\) as \(\pi \colon E \to M\).
Inductively define covariant derivatives of the curvature 
\[
\nabla^p K \colon E \to \LieG \otimes \Lm{2}{\LieG/\LieH}^* \otimes (\otimes^p \LieG^*)
\]
by \(\nabla^0 K = K\) and \(d \nabla^p K = \nabla^{p+1}K \omega\).
Pick a point \(m \in M\) and a point \(e \in E_m \subset E\). 
Take any \(v_1, v_2, w \in T_m M\). 
Take any vectors \(A_1, A_2, B \in \LieG\) such that
\[
\pi'(e) \vec{A}_j = v_j \text{ and } \pi'(e) \vec{B} = w\, ,
\]
where \(\pi' \colon TE\, \rightarrow\, TM\) is the differential of \(\pi\). 
If \(K\left(A_1,A_2\right)=0\) and \(B \in \LieH\), then
\[
\nabla^1 K\left(A_1,A_2\right)B
=\left.\frac{d}{dt}\right|_{t=0} r_{e^{t\vec{B}}} K\left(A_1,A_2\right) =0\, .
\]
Therefore if \(0=v_1 \hook \nabla \omega\) then \(\nabla^1 K\left(A_1,A_2\right)B\) only depends on \(B\) modulo \(\LieH\).
Define \(\nabla^1 \omega\left(v_1,v_2\right)w\) to be the associated element of \(\pr{\amal{E}{H}{\LieG}} \otimes \Lm{2}{T^*M} \otimes T^*M\).
Inductively, if \(0=v_1 \hook \nabla^p\omega\) we define
\[
\nabla^{p+1} \omega\left(v_1, v_2\right)\left(w_1,w_2,\dots,w_{p+1}\right)
=\nabla^{p+1} K\left(A_1,A_2\right)\left(B_1,B_2,\dots,B_{p+1}\right),
\]
which is an element of
\[
\amal{E}{H}{\pr{\LieG \otimes \Lm{2}{\LieG/\LieH}^* \otimes
\bigotimes^p \left(\LieG/\LieH\right)^*}}
=
\pr{\amal{E}{H}{\LieG}} 
\otimes 
\Lm{2}{T^*M} 
\otimes
\bigotimes^p T^*M.
\]

Take \(v_1\) tangent to a rational curve in \(M\). Then
\[
0 = \nabla^p \omega\left(v_1,v_2\right)\left(w_1,w_2,\dots,w_p\right)
\]
for all vectors \(v_2, w_1, \dots, w_p \in T_m M\), for all \(p\)
(see the proof of Lemma~\ref{lemma:TgtKillsCurvature}).
Consequently, if \(A \in \LieG\) projects to some nonzero vector \(v \in T_m M\) tangent to a rational curve,
then \(A \hook K\) vanishes with its derivatives of all orders at every point of the fiber \(E_m\).
Therefore \(A \hook K\) vanishes everywhere.
Let \(\LieH'\) be the set of all vectors \(A \in \LieG\) so that \(A \hook K=0\) everywhere.
Note that \(\LieH \subset \LieH'\) and that \(\LieH'\) is \(H\)-invariant, because \(K\) transforms in the adjoint representation under \(H\)-action.
Let \(H' \defeq H \exp \LieH' \subset G\).
Because \(\LieH\) is \(H\)-invariant, \(H'\) is invariant under conjugation by elements of \(H\).
So \(H'\) is the smallest subgroup of \(G\) containing \(H\) and having Lie algebra \(\LieH'\), and \(H'\) has a countable set of components, no larger than the number of components of \(H\).
So \(H' \subset G\) is a Lie subgroup, perhaps not embedded, with Lie algebra \(\LieH'\) \cite{Yamabe:1950}.
Lemma~\ref{lemma:TgtBundle} identifies \(\amal{E}{H}{\pr{\LieG/\LieH}}=TM\) and thereby identifies \(\amal{E}{H}{\pr{\LieH'/\LieH}} \subset  \amal{E}{H}{\pr{\LieG/\LieH}}\) with a holomorphic vector subbundle \(\mathcal{V} \subset TM\).
The result follows from \cite[p. 15]{Cap:2006} (unnumbered theorem) where our notation \(E, H, H', K\) is \v{C}ap's \(\mathcal{G}, Q, P, \kappa\).
\end{proof}

\begin{corollary}
Suppose that \(H \subset G\) is a closed complex sugroup of a complex Lie groups and the associated Lie algebra \(\LieH \subsetneq \LieG\) is a maximal proper complex Lie subalgebra.
Suppose that \(E \to M\) is a holomorphic \((X,G)\)-geometry on a connected complex manifold, \(X=G/H\), and \(M\) contains a rational curve. 
Then the \((X,G)\)-geometry is flat and locally isomorphic to the model Cartan geometry
on \(X\).
\end{corollary}

For example, on any connected complex manifold containing a rational curve, every holomorphic conformal or projective connection is flat.

\begin{corollary}\label{corollary:ample.rational}
Suppose \((X,G)\) is a complex homogeneous space and that \(E \to M\) is a holomorphic \((X,G)\)-geometry on a connected complex manifold.
Suppose that \(M\) contains an ample rational curve.  
Then the \((X,G)\)-geometry is flat and locally isomorphic to the model Cartan geometry
on \(X\).
\end{corollary}

\begin{proof}
The ample ambient tangent bundle along the rational curve prevents the rational curve from sitting in the leaves of a foliation.
\end{proof}

\subsection{Families of rational curves and trees}

A \emph{tree of projective lines} is a proper nodal curve of arithmetic genus zero, or in other words, it is a compact, connected and simply connected reduced complex projective curve with only nodal singularities such that each irreducible component is a projective line.
A \emph{rational tree} in a complex manifold \(M\) is a nonconstant morphism \(f \colon T \to M\) from a tree of projective lines.

\begin{definition}[\cite{Bien/Borel/Kollar:1996}, p. 106, Definition 1.1]
A \emph{family of rational trees through a point} \(m_0 \in M\) in a complex manifold is a triple of the form \((Z \to Y\, ,s\, ,f)\), where \(Z \to Y\) is a proper flat morphism of complex spaces such that the generic fiber is a tree of projective lines, \(s \colon Y \to Z\) is a regular section and \(f \colon Z \to M\) is a morphism, such that \(f(s(y))=m_0\) for all \(y \in Y\).
\end{definition}

A \emph{rational chain} of length \(\ell\) in a complex space \(X\) is a rational tree whose irreducible components \(C_1, C_2, \dots, C_{\ell}\) have \(C_i\) only intersecting \(C_1 \cup \dots \cup C_{i-1}\) at a single point of \(C_{i-1}\) and at no point of \(C_1 \cup \dots \cup C_{i-2}\).
Two points \(m\) and \(m'\) of a complex space \(M\) are \emph{chain equivalent} if they lie in the image of a rational chain.
The \emph{rational envelope} \(\RE{m}\) of a point \(m\) is its chain equivalence class.

\begin{lemma}[\cite{Araujo/Kollar:2003} p. 25, \cite{Kollar:1996} p. 115, Cor. 3.5.3, Thm. 7.6]\label{lemma:tree.to.curve}
Every rational tree with marked point \(m_0\) in a convex complex manifold \(M\) has image arising as a fiber in a flat family of rational curves in \(M\) all passing through \(m_0\).
\end{lemma}

\begin{lemma}\label{lemma:tree.image}
In any flat proper family of rational trees in a complex manifold \(M\), every fiber has the same image in \(M\) as some rational tree.
\end{lemma}
\begin{proof}
In a flat family of rational trees, some fiber might have nonnodal singularities, failing thereby to be a rational tree.
By Grauert's theorem \cite[p. 207 Thm. 6]{Grauert/Remmert:1984} \cite[p. 10, Thm. 2.3]{Bell/Narasimhan:1990}, the dimension, topological Euler characteristic, number of components and arithmetic genus are preserved in a flat family of compact complex curves.
Therefore every irreducible component is the image of a rational curve.
The fiber is the image of a proper curve consisting of projective lines, perhaps with nonnodal singularities.
By introducing extra irreducible components so that the nonnodal singularities are resolved to nodal singularities, every fiber is thus the image of a rational tree, perhaps in many ways.
\end{proof}

\section{Rational envelopes in convex manifolds}\label{section:freedom}

We received essential assistance in writing the following proof from J\'anos Koll\'ar.

\begin{proposition}\label{prop-kollar}
Let \(M\) be a tame convex compact complex manifold. 
Every rational envelope in \(M\) is an irreducible closed analytic subvariety of \(M\).
For each point \(m_0 \in M\), there is a flat family of rational trees \(Z_{m_0} \to Y_{m_0}\) whose generic nonempty fiber over the Douady space is a rational chain through \(m_0\), and whose every fiber is the image of a rational chain through \(m_0\), so that these fibers reach every point of the rational envelope \(\RE{m_0}\), i.e. \(Z_{m_0} \to \RE{m_0}\) is onto.
\end{proposition}

\begin{proof}
The idea of the proof of the proposition is that of Theorem 4.13 in \cite[p. 215]{Kollar:1996}.
Take any point \(m_0 \in M\). 
Any point of the rational envelope \(\RE{m_0}\) is connected to \(m_0\), so \(\RE{m_0}\) is connected.
Let \(Y\,\subset\, M\) be a maximal closed irreducible analytic subvariety of \(M\) contained in \(\RE{m_0}\).

Suppose that we can prove that for any rational curve \(g \colon \Proj{1} \to M\), if \(m_0 \in g\of{\Proj{1}}\) then \(g\of{\Proj{1}}\subset Y\).
If \(\RE{m_0}\) contains two maximal closed irreducible analytic subvarieties \(Y_1\, , Y_2\, \subset\, M\) such that \(m_0' \in Y_1\bigcap Y_2\) and \(m_0 \notin Y_1\bigcap Y_2\), then replace \(m_0\) by \(m_0'\).
We then have every rational curve through \(m_0\) contained in both \(Y_1\) and \(Y_2\), so \(\RE{m_0} \subset Y_1 \cap Y_2 \subset Y_1 \cup Y_2 \subset \RE{m_0}\) so \(Y_1=Y_2\).
Therefore \(\RE{m_0}\) is an irreducible closed analytic subvariety of \(M\). 

Suppose that there is a rational curve \(g \colon \Proj{1} \to M\) with image not contained in \(Y\) and with \(m_0 \in g\of{\Proj{1}}\);  we look for a contradiction.
Let \(\Gamma \subset \Proj{1} \times M\) be the graph of that rational curve.
By lemma~\ref{lemma:UnobstructedDeformation}, \(\Gamma\) is a smooth point of the Douady space of compact complex subspaces of \(\Proj{1} \times M\), so lies in a unique component \(D\).
Denote by \(Z \subset D \times \Proj{1} \times M\) the universal flat family over \(D\).
By convexity of \(\Proj{1} \times M\), every rational curve in \(M\) is free, and its graph is reduced and irreducible, so the generic fiber of \(Z \to D\) is reduced and irreducible.
By tameness of \(M\), \(Z\) and \(D\) are compact.

Let \(f \colon Z \to M\) be the composition \(Z \subset D \times \Proj{1} \times M \to M\).
Let \(Z_Y \defeq f^{-1}Y \subset Z\) be the points of \(Z\) at which the marked point lies in \(Y\).
Since \(M\) is convex, \(Z\) is a complex manifold near \(\Gamma\), and \(f\) is a smooth morphism near \(\Gamma\) \cite[p. 115, Corollary 3.5.4]{Kollar:1996}. 
Hence the restriction
\[
\left.f\right|_{Z_Y} \colon Z_Y \to Y
\]
is dominant. 
Therefore, for every general \(m \in Y\), there is a deformation of \(g\) passing through \(m\) \cite[p. 113 Cor. 2.12]{SCV:VII}, and some curve \(\Gamma' \subset \Proj{1} \times M\) arbitrarily close to \(\Gamma\) has image in \(M\) striking a smooth point of \(Y\).
A Zariski open subset of \(Z\) consists of graphs of morphisms \(h \colon \Proj{1} \to M\) with marked point \(z \in \Proj{1}\) \cite[p. 87]{Douady:1966}.
The point \(\Gamma\) is such a graph.
Therefore every point near \(\Gamma\) is also such a graph.
So we can suppose that \(\Gamma'\) is also a graph, i.e. that \(m \in Y\) is a smooth point and some rational curve \(h \colon \Proj{1} \to M\) passes through \(m\) without lying in \(Y\).
The subset of \(Z_Y\) consisting of pointed graphs \((h,z)\) of a morphism \(h \colon \Proj{1} \to M\) with a point \(z \in \Proj{1}\) and \(h(z) \in Y\) is Zariski open and nonempty, so dense in \(Z_Y\), and lies entirely in the smooth points of \(Z\).

By unobstructed deformation, the tangent space of \(Z_Y\) at any such \((h,z)\) consists of the vectors \((\dot{h},\dot{z})\) composed of global sections \(\dot{h} \in \cohomology{0}{\Proj{1},h^*TM}\) and vectors \(\dot{z} \in T_z \Proj{1}\) so that \(\dot{h}(z)+h'(z)\dot{z} \in T_{h(z)} Y\).
By freedom, and the classification of vector bundles on the projective line \cite[p. 122, Th\'eor\`eme 1.1]{Grothendieck:1957}, at every point \(w \in \Proj{1}\), the values \(\dot{h}(w)\) of these sections \(\dot{h}\) have the same dimension as the values of \(\dot{h}(z)\).
By the implicit function theorem, we can move the point \(h(w)\) by small deformations of \(h\) to sweep out a smooth complex manifold of at least as large a dimension as that of \(Y\).
By Remmert's proper mapping theorem \cite[p. 5, Thm. 1.1]{Bell/Narasimhan:1990} the compact irreducible component of \(Z_Y\) containing the point \((h,z)\) is mapped to a compact irreducible subvariety \(Y' \subset M\) containing this complex manifold of values \(h(w)\).
This \(Y'\) contains a dense open subset of \(Y\), and is closed, so contains \(Y\), but is not equal to \(Y\).
But \(Y\) is irreducible, so \(Y'\) has larger dimension than \(Y\).

Every element \(\Gamma \in D\) lies in a flat family of subvarieties whose generic element is a rational curve in \(\Proj{1} \times M\).
So every element of \(D\) has the same image in \(M\) as some rational tree in \(\Proj{1} \times M\) by Lemma~\vref{lemma:tree.image}.
This rational tree is a connected finite union of rational curves in \(\Proj{1} \times M\), so if one point of the tree strikes \(\Proj{1} \times \RE{m}\) then the image of the tree lies in \(\Proj{1} \times \RE{m}\).
Therefore the irreducible subvariety \(Y'\) lies in \(\RE{m}\).
But \(Y\) is maximal among subvarieties in \(\RE{m}\) containing \(m\), a contradiction.

For any smooth point \(m \in Y\), every rational curve through \(m\) lies in \(Y\).
Take a point \(m \in Y\), perhaps not smooth, and a rational curve \(g \colon \Proj{1} \to M\) so that \(g(0)=m\).
Every deformation of \(g\) touching a smooth point of \(Y\) lies entirely in \(Y\).
But there are deformations arbitrarily close to \(g\) passing through smooth points of \(Y\), so \(g\) lies entirely in \(Y\).

By the same process, after a finite number of steps, the dimension of the set swept out by the universal flat family of chains, generically of length \(\ell\), containing a point \(m_0 \in M\) stops increasing with \(\ell\), and is a compact subvariety of \(M\). 
Therefore, the generic point of \(\RE{m_0}\) lies along some rational chain of length at most \(\ell\).
Chains of that length passing through \(m_0\) live in a proper flat family, the image of which is a compact subvariety of \(\RE{m_0}\), again by the Remmert proper mapping theorem, so is all of \(\RE{m_0}\) because \(\RE{m_0}\) is irreducible.
\end{proof}

For a pair of general points of \(\RE{m_0}\), \cite[p. 215, Theorem IV.4.13]{Kollar:1996} gives an explicit  bound on the length of rational chains connecting them.

We are not aware of an example of a convex compact connected complex manifold \(M\) and point \(m_0 \in M\) where \(\RE{m_0}\) is not smooth or not simply connected, but we have no proof that \(\RE{m_0}\) is smooth or simply connected.
For generic \(m \in M\), the rational envelope \(\RE{m}\) is a fiber of the MRC fibration and thus smooth and rationally connected \cite{Kollar:18/March/2014}.
Campana's theorem \cite[p. 545, Theorem 3.5]{Campana:1991} thereby proves that \(\RE{m}\) is simply connected for generic \(m\).
See the discussion of Campana's theorem in section~\ref{section:structure.group} below.
When we apply Proposition~\ref{prop-kollar}, we have to prove that the particular rational envelopes arising in our applications are smooth.

\section{Families of rational trees in Cartan geometries}

Take a complex homogeneous space \((X,G)\) and a holomorphic \((X,G)\)-geometry \(E \to M\).
Take a family of rational trees \(\pr{Z \to Y, s, f_0}\) through a point \(m_0 \in M\) and a point \(e_0 \in E_{m_0}\).
The \emph{developing map} of the family with frame \(e_0\) is the developing map \(f_1 \colon Z \to M_1\) as defined in subsection~\ref{subsection:Development}, for the constant section \(y \mapsto e_0\).

\begin{lemma}\label{lemma:re.flat}
Suppose that \(M\) is a complex manifold bearing a holomorphic Cartan geometry.
The curvature of the Cartan geometry vanishes on the tangent space of any smooth point of the rational envelope of any point of \(M\).
\end{lemma}

\begin{proof}
The rational envelope of any point of \(M\) is rationally connected by definition.
Any rational curve lies inside a single leaf of the foliation of Theorem~\ref{theorem:local}. 
Therefore each rational envelope lies inside a single leaf of that foliation. 
But the curvature vanishes on each leaf.
\end{proof}

\begin{proposition}\label{proposition:DeformationIdentification}
Suppose that \((X,G)\) is a complex homogeneous space and that \(M\) is a complex manifold bearing a holomorphic \((X,G)\)-geometry.
Take a rational tree \(T\) through \(m_0 \in M\) and develop it to a rational tree \(T'\) through \(x_0 \in X\).
Every sufficiently small deformation of \(T'\) is also the development of a unique deformation of \(T\).
If \(M\) is tame then every deformation of \(T'\) is the development of a unique deformation of \(T\).
 \end{proposition}

\begin{proof}
Since \(M\) is convex and compact, any rational tree in \(M\) through \(m_0\) is the fiber of a flat family of rational curves through \(m_0\) and the same is true for \(X\) for the same reason: lemma~\ref{lemma:tree.to.curve}.
It suffices to prove the local result for rational curves and take limits in a flat family, since the limits are uniform with derivatives at smooth points, so development extends to the limits. 
We can take the flat family to be the universal flat family over an irreducible component of the Douady space.
By Horikawa's deformation theorem \cite[p. 649]{Horikawa:1974} and convexity of \(M\), the deformations of any rational curve passing through a chosen point \(m_0 \in M\) are unobstructed and form a smooth moduli space near that curve.
All rational trees in \(M\) through \(m_0\) develop to rational trees in \(X\), and distinct trees develop to distinct trees.
As we vary trees holomorphically, their developments vary holomorphically.

We therefore have an injective morphism between the space of deformations of any rational curve
and the space of deformations of its development.
The deformations of a rational curve passing through a chosen point \(m_0 \in M\) form a complex space of dimension given by the dimension of sections of the ambient tangent sheaf vanishing at \(m_0\). 
By Lemma \ref{remark:TMZeroEqualsTMOne}, the ambient tangent sheaf of a rational curve is identified 
with the ambient tangent sheaf of any of its developments.
Therefore the deformation space of a rational tree and of its development are complex  spaces of equal dimension, with tangent spaces identified by development.
The developing map, being injective and having injective differential, gives a local isomorphism between these spaces near any smooth point.
By the same argument, the developing map is an injective local isomorphism on every stratum of the Douady component.
The developing map gives a global isomorphism on each compact component by properness, and therefore on the union of all compact components.
Every rational curve is a smooth point in a unique component.
If \(M\) is tame, this component is compact.
Therefore the developing map extends to all rational curves, if \(M\) is tame, and then extends to the limits of the rational curves, i.e. to all rational trees.
\end{proof}

Pick a family of rational trees \(f \colon Z\,\rightarrow \, M\) through a point \(m\,\in\, M\). 
If the image of \(f\) is dense in the analytic Zariski topology in \(\RE{m}\), \(f\) is a \emph{sufficiently large family} through \(m\).
By Proposition~\ref{prop-kollar}, there is a sufficiently large family for each point \(m\) in any tame convex compact complex manifold, and any family of rational trees through a chosen point of \(M\) whose image has large enough dimension is sufficiently large.

\begin{lemma}\label{lemma:GPrimeContainsHNought}
Suppose that \((X,G)\) is a complex homogeneous space and that \(E \to M\) is a holomorphic \((X,G)\)-geometry on a tame compact complex manifold.
Pick a point \(m_0 \in M\) and a frame \(e_0 \in E_{m_0}\).
If a point \(gx_0 \in X\) lies on the development of a rational tree through \(m_0\), then every point \(hgx_0\) lies on the development of a rational tree through \(m_0\), for any point \(h\) in the identity component of \(H\).
\end{lemma}

\begin{proof}
Pick a rational tree \(f \colon T \to X\) rooted at \(x_0 \in X\) which is the development of a rational tree in \(M\) through \(m_0\). 
Let \(H^0\) be the identity component of \(H\).
Consider the family \(F \colon T \times H^0 \to X\),
\(F(t,h)=hf(t)\).
By Proposition~\ref{proposition:DeformationIdentification}, every element of this family of rational trees  is the development of a rational tree through \(m_0\).
\end{proof}

Suppose \((X,G)\) is a complex homogeneous space and that \(E \to M\) is a holomorphic \((X,G)\)-geometry on a complex manifold.
Fix a point \(m_0 \in M\) and a frame \(e_0 \in E_{m_0}\).
Pick a rational tree \(f_0 \colon T \to M\) through \(m_0\), and let \(f_1 \colon T \to X\) be its development with frame \(e_0\).
The points of \(f_1^* G\) can be written as pairs \(\left(t,g\right)\) where \(t \in T\)
and \(g \in G\) and \(f_1(t) = g x_0\).
Let \(H'=H'\left(e_0\right)\) be the set of elements \(g \in G\) for which there is some rational tree \(f_0\) through \(m_0\) with development \(f_1\) with frame \(e_0\) and with a point \(\left(t,g\right) \in f_1^* G\). 
Call \(H'\) the \emph{drop group}; lemma~\ref{lemma:GPrimeContainsHNought} implies that \(H \subset H'\).

The bundle \(f_1^* G\) has a distinguished point \(\left(o,1\right)\) where \(o \in T\) is the marked point with \(f_0\left(o\right)=m_0\).
There is a distinguished path component of \(f_1^* G\), the path component of the distinguished point.

Let \(H'_0=H'_0\left(e_0\right)\) be the set of elements \(g \in G\) for which there is some rational tree \(f_0\) through \(m_0\) (whose development we denote \(f_1\)), and there is some point \(\left(t,g\right) \in f_1^* G\) in the distinguished path component of \(f_1^* G\). 
Call \(H'_0\)  the \emph{restricted drop group}. 
Clearly \(H'_0 \subset H'\) and the identity component of \(H\) lies in \(H'_0\).
We will eventually show that \(H'_0\) is the identity component of \(H'\).

\begin{lemma}
The drop group of a holomorphic Cartan geometry on a tame compact complex manifold is closed under multiplication, as is the restricted drop group.
\end{lemma}

\begin{proof}
Suppose that \((X,G)\) is a complex homogeneous space and that \(E \to M\) is a holomorphic \((X,G)\)-geometry on a compact complex manifold.
Fix a point \(m_0 \in M\) and a frame \(e_0 \in E_{m_0}\).
Let \(H'\) be the drop group of \(e_0\).
Take two rational trees \(f_1 \colon T_1 \to X\) and \(f_2 \colon T_2 \to X\),
both rooted at the point \(x_0 \in X\).
So each has a marked point, say \(p_1 \in T_1\) and \(p_2 \in T_2\)
so that \(f_1\left(p_1\right)=f_2\left(p_2\right)=x_0\).
We want simply to imagine sliding the tree \(T_1\) along the tree \(T_2\), so that the root of \(T_1\) traces out the points of \(T_2\).
If we can do this, then for any point in \(H'\) coming from \(T_1\), say \(g_1\), and any point in \(H'\) coming from \(T_2\), say \(g_2\), we will slide over to find that \(g_1 g_2 \in H'\).

Let \(Y = f_2^* G\). 
The points of \(Y\) have the form \(\left(q_2,g_2\right)\) with \(q_2 \in T_2\) and \(g_2 \in G\)
so that \(f_2\left(q_2\right)=g_2 x_0\).
Let \(Z_1 = Y \times T_1\) and \(Z_2 = Y \times T_2\), trivial bundles over \(Y\). Write points of \(Z_1\) as \(\left(q_2,g_2,r_1\right)\) with \(r_1 \in T_1\), and write points of \(Z_2\)
as \(\left(q_2,g_2,r_2\right)\) with \(r_2 \in T_2\).
Map \(F_1 \colon Z_1 \to X\) by 
\[
F_1\left(q_2,g_2,r_1\right)=g_2 f_1\left(r_1\right).
\]
Map \(F_2 \colon Z_2 \to X\) by 
\[
F_2\left(q_2,g_2,r_2\right)=f_2\left(r_2\right).
\]

Say that
\[
\left(q_2, g_2, r_1\right) \sim \left(q_2,g_2, r_2\right)
\]
if \(r_1=p_1\) and \(r_2=q_2\).
Let \(Z = \left(Z_1 \sqcup Z_2\right)/\sim\), i.e., graft the root of tree \(T_1\) to the point \(q_2\).
We can identify \(Z\) with the variety \(Z \subset T_1 \times T_2 \times Y\) consisting of points \(\left(r_1,r_2,q_2,g_2\right)\) so that \(r_1=p_1\) or \(r_2=q_2\).
Therefore \(Z \to Y\) is a proper family. 
The algebraic variety \(Z\) splits into two irreducible components, say \(Z=Z_1 \cup Z_2\), given by \(Z_1 = \left( r_2 = q_2 \right) \) and \(Z_2 = \left(r_1=p_1\right) \).
Note that \(Z_1 \cap Z_2 \cong Y\). 
Away from the ``grafting points'' \( \left(p_1,q_2,q_2,g_2 \right) \) the map \(Z \to Y\) is locally identified with either \(Z_1 \to Y\) or \(Z_2 \to Y\), so is flat. 
Near the ``grafting points'' we have an exact sequence
\[
0 \to \mathcal{O}_{z,Z} 
\to \mathcal{O}_{z_1,Z_1} \oplus \mathcal{O}_{z_2,Z_2}
\to \mathcal{O}_{y,Y} 
\to 0,
\]
with the middle and final entries being flat over \(\mathcal{O}_{y,Y}\).
Therefore the first entry is flat over \(\mathcal{O}_{y,Y}\) \cite[p. 254, Proposition 9.1A]{Hartshorne:1977}, so \(Z \to Y\) is flat.

The maps \(F_1\) and \(F_2\) agree along the identified points so descend to a morphism \(F \colon Z \to X\) by
\[
\left.F\right|_{Z_1} = F_1 \text{ and } \left.F\right|_{Z_2} = F_2.
\]
We take as section \(s \colon Y \to Z\) the mapping 
\[
s\left(q_2,g_2\right) = \left(p_1,p_2,q_2,g_2\right).
\]
Clearly we have constructed a family of rational trees through \(x_0\).

If \(f_1\) and \(f_2\) are developments of trees in \(M\), then each tree in the family \(F\) is a development of a tree in \(M\), by Proposition~\ref{proposition:DeformationIdentification}, 
so \(F\) is the development of a family of trees in \(M\).
\end{proof}

\begin{lemma}
The drop group is closed under taking inverses, as is the restricted drop group.
\end{lemma}

\begin{proof}
Suppose that \((X,G)\) is a complex homogeneous space and that \(E \to M\) is a holomorphic \((X,G)\)-geometry on a tame compact complex manifold.
Fix a point \(m_0 \in M\) and a frame \(e_0 \in E_{m_0}\).
Let \(H'\) be the drop group of \(e_0\). 
Take a rational tree \(f \colon T \to X\), say with root at \(p_0 \in T\).
Let \(Y = f^* G\), and let \(Z = T \times f^* G\),
with the obvious map \(Z \to Y\). Map
\(h \colon Z \to X\) by \(h\left(p,q,g\right)=g^{-1} f(p)\).
Then \(h\left(p,p_0,1\right)=f(p)\), so
\(h\) deforms \(f\). 
If some point \(g_1 x_0\) lies in the image of \(f\), say \(g_1 x_0 = f\left(p_1\right)\),
then \(p \mapsto h\left(p,p_1,g_1\right)\) is a tree rooted at \(p_1\) with \(h\left(p_1,p_1,g_1\right)=x_0\) and \(h\left(p_0,p_1,g_1\right)=g_1^{-1} x_0\). 
If \(g_1\) lies in the identity component of \(G\), then this family deforms \(f\) into a curve through \(g_1^{-1} x_0\), so \(g_1^{-1} \in H'\).

If \(f\) is the development of a tree in \(M\), then each tree in the family is a development of a tree in \(M\), by Proposition~\ref{proposition:DeformationIdentification}.
Therefore the restricted drop group \(H_0'\) is closed under taking inverses.
Therefore the group generated by \(H \cup H_0'\) lies in \(H'\).

The fiber of \(f^*G\) over \(p_0\) is \(H\).
Every element of any fiber of \(f^*G\) is therefore path connected to \(H\) inside \(G\).
So any such element is taken by left \(H\)-multiplication into the identity component of \(G\).
Therefore every element of \(H'\) is a product of an element of \(H\) and an element of the identity component of \(G\).
Consequently, \(H'\) is closed under taking inverses.
\end{proof}

Summing up: suppose that \((X,G)\) is a complex homogeneous space \(X=G/H\).
Let \(H^0 \subset H\) be the identity component. 
Suppose that \(M\) is a complex manifold with a holomorphic \((X,G)\)-geometry.
Pick a point \(m_0 \in M\) and a frame \(e_0 \in E_{m_0}\).
Let \(H_0'\) be the restricted drop group of \(e_0\) and let \(H'\) be the drop group of \(e_0\).
Then \(H^0 \subset H'_0 \subset H' \subset G\) are subgroups, and \(H^0 \subset H \subset H' \subset G\) are subgroups.

\section{The structure group and the rational envelope}\label{section:structure.group}

Let us look at Campana's theorem \cite[p. 545, theorem 3.5]{Campana:1991} and translate the notation.
Campana's definition of rational connectivity is more restrictive than ours.
A variety on which any two points lie on a rational chain might not be rationally connected in Campana's sense; a variety is \emph{rationally connected in the sense of Campana} if there is a proper family of rational chains covering the entire variety \(Z\), all passing through the same point.
To employ Campana's notation, take a connected normal irreducible analytic space \(Z\), and let \(A\defeq \left\{z_0\right\}\) be a point \(z_0 \in Z\), and let \(S\) be a subvariety of the Doaudy space of deformations of a rational chain through \(z_0\), so that these curves reach every point of \(Z\).
But Campana refers in his definition of rational connectivity in \cite[Definition 3.1 p. 544]{Campana:1991} back to \cite[Notation 2.0 p. 543]{Campana:1991}, where he requires \(S\) to be compact. 
In \cite[Remark 3.2 (4) p. 545]{Campana:1991}, Campana points out that these hypotheses include Fujiki manifolds.
In \cite[Theorem 3.5, p. 545]{Campana:1991}, he then says that under these hypotheses, if \(Z\) is smooth, then \(Z\) is simply connected.
Smoothness of the compact rationally connected variety is essential to Campana's result \cite[Theorem 3.5, p. 545]{Campana:1991} \cite{Kollar:18/March/2014,Kollar:20/March/2014}.
To sum up:
\begin{lemma}[Campana \cite{Campana:1991} p. 545, theorem 3.5]\label{lemma:Campana}
If \(M\) is a compact complex manifold, and some compact subvariety of the Douady space of rational chains through a point of \(M\) contains curves which cover all of \(M\), then \(M\) is simply connected.
\end{lemma}

It is convenient to restate lemma~\ref{lemma:Campana} in the context of convexity.
\begin{lemma}\label{lemma:Campana.prime}
If \(M\) is a tame convex compact complex manifold connected by rational curves then \(M\) is simply connected.
\end{lemma}
\begin{proof}
Since \(M\) is connected by rational curves, \(\RE{m_0}=M\) for every point \(m_0 \in M\).
By Proposition \ref{prop-kollar}, a sufficently large family of rational chains through \(m_0\) covers \(M\); apply lemma~\ref{lemma:Campana}.
\end{proof}

We need a slight variation on the Borel--Remmert theorem \cite[p. 216]{Akhiezer:1986}.

\begin{proposition}\label{proposition:RationallyConnectedHomogeneous}
If \(X\) is a tame rationally connected homogeneous manifold then \(X\) is a rational homogeneous variety \(X=G/P\).
In particular, \(X\) is simply connected.
\end{proposition}

\begin{proof}
Every connected compact homogeneous complex manifold is a homogeneous holomorphic bundle over some flag variety with compact holomorphically parallelizable fibers \cite[p. 216]{Akhiezer:1986}.  

Because \(X\) is homogeneous, \(TX\) is spanned by global holomorphic sections, so every rational curve on \(X\) is free.
By lemma~\ref{lemma:Campana.prime}, \(X\) is simply connected.
We can assume \cite[pp. 3--4]{Wang:1954} that \(X\) is acted on transitively, effectively and holomorphically by a complex semisimple Lie group; let \(G\) denote this group, so that \(F \to X \to G/P\) is a holomorphic fiber bundle for some parabolic subgroup \(P \subset G\).
By \cite[p. 545, Proposition 3.4]{Campana:1991}, \(\cohomology{r}{X,\mathcal{O}}=0\) for \(r>0\) and \(\cohomology{r}{G/P,\mathcal{O}}=0\) for \(r>0\).
The fiber has \(\cohomology{r}{F,\mathcal{O}}>0\) for \(r\) up to the dimension of \(F\), since \(F\) has holomorphically parallel tangent bundle.
Computing out the exact sequence from \(X\) being an \(F\)-bundle we see that \(F\) has dimension \(0\). 
Therefore \(X\) is a finite unramified covering space of a rational homogeneous variety. 
Every rational homogeneous variety is simply connected \cite[p. 207]{Akhiezer:1986}.
So \(X=G/P\) is a rational homogeneous variety.
\end{proof}

\begin{corollary}\label{corollary:GroupAction}
Suppose that \((X,G)\) is a complex homogeneous space \(X=G/H\) and \(E \to M\) is a holomorphic \((X,G)\)-geometry.
Fix a point \(m_0 \in M\) and a frame \(e_0 \in E_{m_0}\).
Let \(H'_0\) be the restricted drop group of \(e_0\). 
Then \(H'_0\) is a Lie subgroup of \(G\).
The image of the developing map \(\RE{m_0} \to X\) is the homogeneous space \(H'_0/\left(H'_0 \cap H\right)\).
\end{corollary}

\begin{proof}
Every path connected subgroup of a Lie group is a Lie subgroup \cite[p. 39]{Onishchik/Vinberg:1993}.
So \(H'_0 \subset G\) is a Lie subgroup. 
The subgroup \(H'_0\) acts transitively on the set of points in \(X\) which lie on the developments through \(e_0\) of rational trees through \(m_0\), by definition.
By Lemma~\ref{lemma:GPrimeContainsHNought}, the identity component of \(H\) lies in \(H'_0\).

If we take any point \(m \in \RE{m_0}\), we can find a rational tree through \(m_0\) and \(m\).
We can then develop this rational tree to a rational tree through a point \(x_0 \in X\) stabilized by \(H\).
So \(\RE{m_0}\) develops inside \(H'_0 x_0 = H'_0/\left(H'_0 \cap H\right)\).
By definition of \(H'_0\), the set \(H' x_0\) consists precisely of the points lying on developments of rational trees, so precisely of the image of the developing map \(\RE{m_0} \to X\). 
This image is the image of a closed analytic subvariety by Proposition~\ref{prop-kollar}, so is a closed analytic subvariety by Remmert's proper mapping theorem \cite[p. 5, Thm. 1.1]{Bell/Narasimhan:1990}.
The action of \(H_0'\) on this subvariety is holomorphic and transitive, so this subvariety is a complex homogeneous space of \(H_0'\), and \(H_0'\) acts on it with stabilizer \(H_0' \cap H\), which is therefore a closed subgroup of \(H_0'\).
\end{proof}

\begin{lemma}\label{lemma:CoveringMap}
Fix a point \(m_0 \in M\) and a frame \(e_0 \in E_{m_0}\). 
The developing map \(\RE{m_0} \to H'_0/\left(H'_0 \cap H\right)\) is an isomorphism of varieties.
\end{lemma}

\begin{proof}
The developing map \(\RE{m_0} \to H'_0 x_0\) is surjective by Corollary~\ref{corollary:GroupAction}.
The developing map is locally injective as a map to \(X\) because the curvature vanishes.
The developing map is a local isomorphism, being a surjective and locally injective
morphism of analytic varieties to a smooth analytic variety. 
By Proposition~\ref{prop-kollar}, the space \(\RE{m_0}\) is compact.
Therefore the developing map \(\RE{m_0} \to H'_0 x_0\) is a covering map.
In particular, every rational tree through \(m_0\) in \(\RE{m_0}\) projects to a rational tree through \(x_0\) in \(H'_0 x_0\), and conversely every rational tree in \(H'_0 x_0\) lifts uniquely to a rational tree through \(m_0\) in \(\RE{m_0}\).
Since \(\RE{m_0}\) is rationally connected, it follows that \(H'_0 x_0\) is also rationally
connected.
The Doaudy space of deformations of a rational tree in \(H'_0 x_0\) through \(x_0\) is canonically isomorphic as a complex space to the Doaudy space of deformations of the corresponding rational tree in \(\RE{m_0}\) through \(m_0\).
In particular, every component of that Douady space is proper.

By Proposition~\ref{prop-kollar} there is some proper flat family of rational chains through \(m_0\) whose rational chains cover all of \(\RE{m_0}\).
The corresponding flat family of rational chains through \(x_0\) given by development has rational chains covering all of \(H_0' x_0\).
By lemma~\ref{lemma:Campana.prime}, both \(\RE{m_0}\) and \(H_0 x_0\) are simply connected.
Therefore the covering map \(\RE{m_0} \to H'_0 x_0\) is an isomorphism.
\end{proof}

\begin{lemma}\label{lemma:RationalHomogeneous}
Fix a point \(m_0 \in M\) and a frame \(e_0 \in E_{m_0}\).
Let \(H'_0\) be the drop group of \(e_0\). 
Then \(H'_0\) is a complex Lie group, \(H'_0 \cap H\) is a parabolic subgroup of \(H'_0\) and so \(H'_0/\left(H'_0 \cap H\right)\) is a rational homogeneous variety.
\end{lemma}

\begin{proof}
By Lemma~\ref{lemma:CoveringMap}, the development 
\[
\RE{m_0} \to H'_0 x_0 = H'_0/\left(H'_0 \cap H\right)
\]
is an isomorphism with a compact, simply connected, rationally connected complex manifold, so \(H'_0 x_0\) is a compact, rationally connected, homogeneous complex manifold.
By Proposition~\ref{proposition:RationallyConnectedHomogeneous}, \(H'_0 x_0 = H'_0/\left(H'_0 \cap H\right)\) is a rational homogeneous variety. 
By definition, the group \(H'_0\) is just the identity component in the subgroup of \(G\) preserving this rational homogeneous variety, and so is complex analytic. 
By the Borel--Remmert theorem \cite[p. 216]{Akhiezer:1986}, the subgroup \(H'_0 \cap H \subset H'_0\) is parabolic.
\end{proof}

\section{The new structure group is complex}

\begin{lemma}
Fix a point \(m_0 \in M\) and a frame \(e_0 \in E_{m_0}\).
The drop group \(H'=H'\left(e_0\right)\) of the frame is just the set of elements \(g \in G\) so that \(g\) preserves \(H'_0 x_0\).
In particular, \(H' \subset G\) is a closed Lie subgroup.
\end{lemma}

\begin{proof}
If \(g\) preserves \(H'_0 x_0\) then \(g x_0\) lies on a rational tree through \(x_0\), by definition of \(H'_0\), so \(g \in H'\).

If \(g \in H'\), then \(g x_0\) lies on a rational tree \(f_1 \colon T \to X\) through \(x_0\) which is the development of a rational tree from \(M\), i.e., \(g \in f_1^* G\).
But \(f_1(T) \subset H'_0 x_0\) by definition of \(H'_0\).
So \(g\) lies in the preimage in \(G\) of the space \(H' x_0\). 
So \(g\) lies in the subgroup of \(G\) preserving \(H' x_0\).
\end{proof}

\begin{lemma}
Fix a point \(m_0 \in M\) and a frame \(e_0 \in E_{m_0}\).
The restricted drop group \(H'_0=H'_0\left(e_0\right)\) of the frame is the identity component of the drop group \(H'=H'\left(e_0\right)\).
The drop group is a complex Lie subgroup of \(G\).
\end{lemma}

\begin{proof}
The identity component \(H'_{\text{id}}\) of the drop group preserves \(H'_0 x_0\).
Therefore \(H'_{\text{id}} \subset H'_0\) by definition of \(H'_0\). 
The restricted drop group is a path connected subgroup of the drop group by definition, so \(H'_0 \subset H'_{\text{id}}\). 
Therefore \(H'_{\text{id}} = H'_0\). 
Since \(H'_0 \subset G\) is a complex Lie subgroup, it follows that \(H' \subset G\) is a complex Lie subgroup. 
\end{proof}

\begin{lemma}
Let \(E \to M\) be the principal \(H\)-bundle of a holomorphic Cartan geometry on \(M\).
Then the drop group of a frame is constant as we vary the frame through any connected component of \(E\), as is the restricted drop group.
\end{lemma}

\begin{proof}
Suppose that the model is \((X,G)\) and let \(H \subset G\) be the stabilizer of a point \(x_0 \in X\).
For each point \(m_0 \in M\) and frame \(e_0 \in E_{m_0}\), let \(H'_0\left(e_0\right)\) be the restricted drop group of \(e_0\).
By freedom of rational trees, as we vary \(e_0=e_0(t)\), and correspondingly vary \(m_0=m_0(t)\), we can deform any rational tree \(f \colon T \to M\) into a family \(f^t \colon T \to M\) so that a marked point on it passes through \(m_0(t)\), and develop \(f^t\) via the frame \(e_0(t)\) to  a family of rational trees through \(x_0\).
These trees through \(x_0\) form a family in \(X\). 
One member of that  family is a development via \(e_0(0)\), and so by Proposition~\ref{proposition:DeformationIdentification} the entire family is a development of a family \(g^t \colon T \to M\) via \(e_0(0)\).
Therefore the entire family lies inside the homogeneous space \(H'_0\left(e_0(0)\right) x_0\).
Therefore \(H'_0\left(e_0(t)\right) \subset H'_0\left(e_0(0)\right)\) for any \(t\) in our family.
Since the choice of \(t\) and \(0\) are arbitrary, \(H'_0\left(e_0(t)\right)\) is constant.
\end{proof}

\begin{lemma}
The drop group of a holomorphic Cartan geometry on a connected complex manifold is independent of the choice  of frame.
\end{lemma}

\begin{proof}
Take a complex homogeneous space \((X,G)\) and let \(H \subset G\) be the stabilizer of a point \(x_0 \in X\).
Take a holomorphc \((X,G)\)-geometry \(E \to M\) on a connected complex manifold \(M\).
Pick a rational tree \(f_0 \colon T \to M\) and a frame \(e_0\) and an element \(h \in H\).
Suppose that \(f_1 \colon T \to X\) is  the development of \(f_0\) via the frame \(e_0\).
By left invariance on \(G\) the development of \(f_0\) via the frame \(e_0 h\) is 
\(h^{-1} \, f_1\).
Therefore \[H'\left(e_0h\right)=h^{-1} H'\left(e_0\right)\, .\]
But \(H \subset H'\left(e_0\right)\), so \(H'\left(e_0h\right) = H'\left(e_0\right)\).
Therefore on any fiber of \(E\), the drop group is constant. 
But it is also constant as we move through a connected component of \(E\). 
Since \(M\) is connected, the drop group is constant.
\end{proof}

\section{Quotienting by the new structure group}

\begin{proposition}\label{prpr}
Take a complex homogeneous space \((X,G)\), pick a point \(x_0 \in X\) and let \(H \subset G\) be the stabilizer of \(x_0\).
Take a connected compact tame complex manifold \(M\).
Let \(E \to M\) be the principal \(H\)-bundle of a holomorphic \((X,G)\)-geometry on \(M\), and let \(\omega\) be the Cartan connection form on \(E\). 
Suppose that \(M\) contains a rational curve.
Let \(H'\) be the drop group of the Cartan geometry.
There is a unique holomorphic right \(H'\)-action on \(E\)  so that
\begin{enumerate}
\item
this action extends the right \(H\)-action, and
\item
for any choice of point \(m_0 \in M\) and \(e_0 \in E\), the development of the rational envelope
\(\RE{m_0} \to M\) via the frame \(e_0\) to a map \(\RE{m_0} \to X\) has  associated bundle isomorphism  \(F \colon \left.E\right|_{\RE{m_0}} \to \left.G\right|_{H'x_0}\) which is \(H'\)-equivariant.
\end{enumerate}
\end{proposition}

\begin{proof}
By definition of \(H'\), \(\left.G\right|_{H'x_0}=H'\),
so if there is to be such an \(H'\)-action, it is unique.
Each choice of point \(m_0\) and frame \(e_0\) gives an isomorphism \(F \colon \left.E\right|_{\RE{m_0}} \to \left.G\right|_{H'x_0}\) by development, so defines a holomorphic right \(H'\)-action on \(\left.E\right|_{\RE{m_0}}\) extending the \(H\)-action.
The \(H'\)-action on \(\left.E\right|_{\RE{m_0}}\) varies holomorphically with choice of frame \(e_0\) and of point \(m_0\) by continuity of development of a family of integrals.
We only have to prove that this action depends only on the choice of rational envelope, not the choice of point \(m_0\) and frame \(e_0\).

Suppose we pick two points \(m_0\) and \(m_1\) in the same rational envelope, and two frames \(e_0 \in E_{m_0}\) and \(e_1 \in E_{m_1}\). 
Since \(m_0\) and \(m_1\) lie in the same rational envelope, under development via \(e_0\), \(e_0\) is taken to \(1 \in H'\), while \(e_1\) is taken to some point of \(H'\), say \(g\). By left invariance, development by \(e_1\) is just development by \(e_0\) followed by left translation by \(g^{-1}\). 
Since right action on \(H'\) commutes with left action, the right action of \(H'\) on \(E\)
computed via either development is the same.
\end{proof}

It is not immediately clear whether or not the action of \(H'\) turns \(E \to E/H'\) into a principal \(H'\)-bundle; nonetheless we can see how the Cartan connection behaves.

\begin{lemma}
Under the hypotheses of Proposition \ref{prpr}, conditions (1), (2) and (3) of the definition of a Cartan connection (Definition~\ref{def:CartanConnection}) are satisfied (with \(H, \LieH\) replaced by \(H', \LieH'\) and \(V\) replaced by \(\LieG\)).
\end{lemma}

\begin{proof}
Condition (2): the Cartan connection is a linear isomorphism on each tangent space: the Cartan connection is unchanged from the original Cartan geometry
\(E \to M\). 
Condition (3): the \(H'\)-action is generated by the vector fields \(\vec{A}\) for \(A \in \LieH'\). 
This is true in the model, i.e., on \(G\), and therefore on \(E\) because the \(H'\)-action
is identified with the \(H'\)-action on \(H' \subset G\). 
We only have to check condition (1): that the Cartan connection transforms in the adjoint representation.

The group \(H'\) is generated by \(H'_0\) and \(H\).
Under \(H\)-action, the Cartan connection transforms in the adjoint representation. So it suffices to check if under \(H'_0\)-action the Cartan connection transforms in the adjoint representation. 
By definition \(H'_0\) is connected. 
Therefore it suffices to check whether under the Lie algebra action of the Lie algebra \(\LieH'_0\) of \(H'_0\), the Cartan connection transforms under the adjoint representation. In other words, for every vector \(A \in \LieH'_0\), we need to check that
\[
\LieDer_{\vec{A}} \omega = - \left[A,\omega\right].
\]
By the Cartan equation
\[
\LieDer_{\vec{A}} \omega = \vec{A} \hook d \omega + d
\left (\vec{A} \hook \omega\right),
\]
and using the equation
\[
\nabla \omega = d \omega + \frac{1}{2} \left[\omega,\omega\right],
\]
we see that it suffices to check whether
\[
\vec{A} \hook \left( \nabla \omega - \frac{1}{2}\left[\omega,\omega\right]\right)
=
- \left[A,\omega\right]
\]
i.e., whether
\[
\vec{A} \hook \nabla \omega =0,
\]
which we know already from Lemma~\ref{lemma:re.flat}.
\end{proof}

\begin{lemma}\label{lemma:ComplexStructure}
If \(E \to M\) is a smooth real Cartan geometry then the curvature of \(E \to M\) is a section of \(\Ad(E) \times \Lm{2}{T^*M}\). 
If the model is a complex homogeneous space, then \(E \to M\) is a holomorphic Cartan geometry, for a unique complex structure on \(M\), if and only if the curvature is a complex linear 2-form valued in the complex vector bundle \(\Ad(E)\).
\end{lemma}

\begin{proof}
Let \(\pi \colon E \to M\) be the bundle map.
At each point \(e \in E\), the Cartan connection \(\omega\) gives a real linear isomorphism \(\omega_e \colon T_e E \to \LieG\).
This isomorphism makes \(E\) into an almost complex manifold, using the complex linear structure on \(\LieG\). 
This almost complex structure is integrable just exactly when \(d \omega\) is a multiple
of \(\omega\), or equivalently when the \((2,0)\)-part of \(d \omega\) valued in \(\overline{\omega} \wedge 
\overline{\omega}\) vanishes, or equivalently when the part of the curvature valued in \(\LieG \otimes_{\C{}} \Lm{0,2}{\LieG/\LieH}\) vanishes. 
Should this occur, the 1-form \(\omega\) is then holomorphic just when \(d \omega\) has no \(\omega \wedge \overline{\omega}\) terms, that is the curvature is complex linear,
in \(\LieG \otimes_{\C{}} \Lm{2,0}{\LieG/\LieH}\).

Clearly the 1-form \(\omega+\LieH\) given by \(\omega_e + \LieH \colon T_e E \to \LieG/\LieH\) is semi-basic for \(\pi\).
Let \(\Omega_e\) be the linear map \(\Omega_e \colon T_m M \to \LieG/\LieH\) so that 
\[
\left(\pi'(e) v\right) \hook \omega_e = \Omega_e(v)
\]
for all \(v \in T_e E\), where \(m=\pi(e)\).
Clearly \(\Omega_e\) is a linear isomorphism for each \(e \in E\). Use \(\Omega_e\) to define a complex structure on \(T_m M\). 
We know that \(r_g^* \omega=\Ad(g)^{-1} \omega\) for all \(g \in H\). 
Therefore \(\Omega_e\) transforms by a complex linear isomorphism if we move \(e\) while fixing the underlying point \(m=p(e)\). 
Therefore \(M\) has a unique almost complex structure for which \(\pi \colon E \to M\) is holomorphic. 
But \(\pi\) is a submersion, so the almost complex structure is a complex structure. 
The morphism \(E \to M\) is therefore a holomorphic principal bundle.
\end{proof}

\begin{proposition}\label{proposition:CartanDescent}
Suppose that \(X=G/H\) is a complex homogeneous space and
\(M\) is a connected tame compact complex manifold with a holomorphic \(\pr{X,G}\)-geometry \(E \to M\). 
Pick a point \(m_0 \in M\) and frame \(e_0 \in E_{m_0}\).
Suppose that \(M\) contains a rational curve.
Pick a point \(m_0 \in M\) and frame \(e_0 \in E_{m_0}\).
Let \(H'\) be the drop group of the Cartan geometry and \(X'=G/H'\).
The \(\pr{X,G}\)-geometry on \(M\) drops to an \(\pr{X',G}\)-geometry on a compact complex manifold \(M'\) so that \(M \to M'\) is a holomorphic fiber bundle, with fibers biholomorphic to \(H'/H\).
\end{proposition}

\begin{proof}
Suppose that \(H \subset G\) is the stabilizer of a point \(x_0 \in X\).
By development, we identify each \(H'\)-orbit in \(E\) with an \(H'\)-orbit in \(G\),
and therefore \(H'\) acts freely on \(E\).

Consider an \(H'\)-orbit \(H'e_0 \subset E\).
By development of a rational envelope \(\RE{m_0} \subset M\), we see that the rational envelope is identified with \(H' x_0 \subset X\) and the \(H'\)-orbit \(H'e_0\) identified with \(H'\).
Since \(H'\) contains \(H^0\), \(H'\) is a union of path components in the preimage in \(G\) of \(H' x_0 \subset X\).
Therefore the orbit \(H'e_0 \) is a union of path components of the preimage in \(E\) of \(\RE{m_0} \subset M\).
Therefore \(H' e_0\) is closed in \(E\). Therefore each \(H'\)-orbit in \(E\) is a closed set and a submanifold biholomorphic to \(H'\). 
At this step of our proof, we need a lemma:

\begin{lemma}
Let
\[
 S \,=\, \Set{(e,eg)\,\vert\, e \in E \text{ and } g \in H'} \,\subset\, E \times E\, .
\]
Then \(S\) is a closed subset in \(E \times E\).
\end{lemma}

\begin{proof}
Take a convergent sequence of points \(e_n \to e \in E\) and a convergent
sequence of points \(e_n g_n \to e' \in E\) with \(g_n\) a sequence in
\(H'\). We want to prove that \(e' = eg\) for some \(g \in H'\).
Write the map \(E \to M\) as \(\pi \colon E \to M\). 
Let \(m_n = \pi\left(e_n\right)\), \(m = \pi\left(e\right)\) and
\(m' = \pi\left(e'\right)\).

By definition of \(H'\), we can find a sequence of rational trees
\(f_n \colon T_n \to M\), with \(f_n\) passing through \(m_n\), 
so that there is a path in \(f_n^* E\) from \(e_n\) to \(e_n g_n\).
By Proposition~\ref{prop-kollar}, we can choose all of these trees to live in the same proper flat family over a compact parameter space.
After perhaps slightly perturbing the various \(e_n\) and \(g_n\), we can choose these trees \(f_n\) to have a uniform bound on their number of rational curve components.
We can even demand that every one of these trees has a single component, by Lemma~\vref{lemma:tree.to.curve}.
By the compactness of the components of the Douady spaces, we can therefore replace these trees by a convergent subsequence converging to a rational tree (perhaps with some larger number of rational curve components).
The limit is a rational tree \(f \colon T \to M\) so that \(f^* E\) contains both \(e\) and \(e'\). 
The development of \(f\) via \(e\) lies in \(H' x_0\), so we see that \(e'=eg\) for some \(g \in H'\).
\end{proof}

Returning to the proof of the theorem, \(S\) is an immersed submanifold of \(E \times E\), since the map \((e,g) \mapsto (e,eg)\) is a local biholomorphism to \(S\). Moreover, this map is 1-1, since \(H'\) acts freely. 
Therefore \(S \subset E \times E\) is a closed set and a submanifold. 
Therefore \(E/H'\) admits a unique smooth structure as a smooth manifold for which \(E \to
E/H'\) is a smooth submersion \cite[p. 262]{Abraham/Marsden:1978}.

Let \(M'=E/H'\). 
Because \(E \to M'\) is a smooth submersion, we can construct a local smooth section near any point of \(M'\), say \(s \colon U \to E\), for some open set \(U \subset M'\).
 Let \(\rho \colon E \to M'\) be the quotient map by the \(H'\)-action. 
Define \(\phi(u,g)=s(u)g\), \(\phi \colon U \times H' \to \rho^{-1}U\).
This map \(\phi\) is a local diffeomorphism, because \(\phi\) is \(H'\)-equivariant and \(\rho \circ \phi(u,g)\,=\,u\) is a submersion.
Moreover, \(\phi\) is 1-1 because \(H'\) acts freely. 
Every element of \(\rho^{-1}U\) lies in the image of our section \(s\) modulo \(H'\)-action, so \(\phi\) is onto. 
Therefore \(E \to M'\) is locally trivial, so a principal \(H'\)-bundle.
The Cartan connection of \(E \to M\) is a Cartan connection for \(E \to M'\). 
By Lemma~\ref{lemma:ComplexStructure}, this Cartan geometry is holomorphic.
\end{proof}

\section{Proof of the main theorem}%
\label{section:main.proof}

\begin{lemma}
Suppose that \(M \to M'\) is a holomorphic fiber bundle with rational homogeneous fibers and that \(M\) is tame, Fujiki, K\"ahler, projective.
Then \(M'\) is tame, Fujiki, K\"ahler, projective.
\end{lemma}
\begin{proof}
A manifold \(M\) is Fujiki just when it is dominated by a K\"ahler manifold, but then \(M\) dominates \(M'\) so \(M'\) is Fujiki.
If \(M\) is K\"ahler, with K\"ahler form \(\omega\), and has fibers of complex dimension \(s\) then integrating \(\omega^{s+1}\) over the fibers of \(M \to M'\) gives a K\"ahler form \(\omega'\) on \(M'\) by Fubini's theorem.
If \(\omega\) is a  K\"ahler form in the image of an integer cohomology class, then the form \(\omega'\) is also in an integer cohomology class, by Poincar\'e duality.
Therefore if \(M\) is projective then \(M'\) is projective by Kodaira's embedding theorem.
Suppose that \(M\) is tame.
Take a rational curve \(f \colon \Proj{1} \to M'\).
The pullback \(f^* M \to \Proj{1}\) is a bundle of rational homogeneous varieties, say isomorphic to \(X=G/P\) where \(G\) is the biholomorphism group of \(X\) and \(P \subset G\) is a parabolic subgroup.
Every parabolic subgroup \(P\) contains a Cartan subgroup of \(G\): \(C \subset P\) \cite[p. 384]{Fulton/Harris:1991}.
The associated holomorphic principal \(G\)-bundle over \(M'\) reduces to a holomorphic principal \(C\)-bundle over \(M'\) \cite[p. 122, Thm. 1.1]{Grothendieck:1957}.
But \(C\) fixes a point of \(X\) so \(f^*M \to \Proj{1}\) admits a global holomorphic section.
In particular, some rational curve in \(M\) maps onto our rational curve in \(M'\), i.e. all rational curves lift.

Let \(D'\) be the component of the Douady space of curves in \(\Proj{1} \times M'\) containing the graph of \(f\) and \(D\) the component of the Douady space of curves in \(\Proj{1} \times M\) containing the graph of a lift of \(f\).
Then \(D \to D'\) gives a morphism of proper flat families, mapping onto the dense open set in \(D'\) consisting of graphs of morphisms \(\Proj{1} \to M'\).
The irreducible components of \(D\) are compact, so they have compact image in \(D'\).
In particular, since the graphs lie in the smooth locus, they lie in a compact irreducible component of \(D\) mapped onto a compact irreducible component of \(D'\).
\end{proof}

We now prove Theorem~\ref{thm:OneCurveForcesDescent}.
\begin{proof}\label{proof:thm:OneCurveForcesDescent}
By Proposition~\ref{proposition:CartanDescent}, if there is a rational curve in \(M\), then we can drop, say to some geometry \(E \to M'\). 
Since every rational curve lies in a rational envelope in \(M\), all rational curves in \(M\) lie in the fibers of \(M \to M'\). 
If \(M'\) contains a rational curve, we can repeat this process, dropping to some geometry \(E \to M''\), etc. 
At each step, we reduce dimension by at least one, i.e., \(\dim M' \le \dim M - 1\), since the connected components of the fibers of \(M \to M'\) are rational homogeneous varieties containing rational curves. 
Therefore after finitely many drops, we arrive at a drop which contains no rational curves, say \(M \to \overline{M}\), say a drop to a \(G/\overline{H}\)-geometry.

At each stage in the process, the connected components of the fibers of \(M \to M'\), and of \(M' \to M''\), etc. are rational homogeneous varieties, and in particular are rationally connected.
So the connected components of the fibers of \(M \to \overline{M}\) are 
rationally connected and compact. 
These fibers are homogeneous spaces, being copies of \(\overline{H}/H\). 
Therefore the fibers are rational homogeneous varieties by Proposition~\ref{proposition:RationallyConnectedHomogeneous}.
The fibers of \(M \to \overline{M}\) lie inside the rational envelopes of \(M\), and therefore lie on rational curves in \(M\), so \(\overline{M} = M'\), i.e., we need precisely one step in this process to achieve the result that \(M'\) contains no rational curves.

Suppose that \(M \to M''\) is some other drop, and that \(M''\) contains no rational curves. The fibers of \(M \to M'\) are rational homogeneous varieties, so rationally connected \cite[p. 34, Corollary 6.9.1]{Kollar:1996}. 
Each fiber of \(M \to M'\) lies inside a fiber of \(M \to M''\), since otherwise a rational curve in a fiber of \(M \to M'\) would project to a rational curve in \(M''\). Therefore the map \(M \to M''\) factors as a map \(M \to M' \to M''\).
Looking upstairs at \(E\), we have \(M=E/H, M'=E/H', M''=E/H''\), so \(H''\) contains \(H'\), i.e., the drop \(M \to M''\) factors as drops \(M \to M' \to M''\).
\end{proof}

\section{Meromorphic maps and rational curves}%
\label{section:mero.maps}

\begin{lemma}%
\label{lemma:mod}
Suppose that \(f \colon X \to Y\) is a modification of reduced complex spaces with \(Y\) smooth.
Denote by \(E\) the exceptional divisor of \(f\).
Through the generic point of every irreducible component of \(E\), there is a rational curve lying in \(E\), mapping to a single point of \(Y\), so that the deformations of this rational curve through any flat family stay in \(E\).
In particular, such a rational curve is not free in \(X\).
\end{lemma}
\begin{proof}
Replace \(X\) with its normalization \cite[p. 161 Thm. 2]{Grauert/Remmert:1984} to ensure that \(X\) is smooth in codimension 1.
Since \(E\) has codimension 1, we can replace \(Y\) with an open subset to assume that both \(X\) and \(E\)  are smooth and irreducible.

By Hironaka's Chow lemma \cite[p. 293, Cor. 2.9]{SCV:VII}, after perhaps restricting to smaller open sets, we can find closed subvarieties \(X_0\defeq E \subset X\) and \(Y_0 \subset Y\) so that \(X_0=f^{-1}Y_0\) so that the blowup \(p \colon \widetilde{X} \to X\) along \(X_0\) is  isomorphic to the blowup \(p \colon \widetilde{Y} \to Y\) along \(Y_0\): \(\widetilde{X} \cong \widetilde{Y}\).
Let \(\widetilde{X}_0=p^{-1}X_0 \subset \widetilde{X}\) and \(\widetilde{Y}_0=p^{-1}Y_0 \subset \widetilde{Y}\) be the exceptional divisors of the blowups.
{\footnotesize
\[
\begin{tikzcd}
& \widetilde{X}_0 \arrow[swap]{dl}{\sim}  \arrow{rr} \arrow{dd}
& & \widetilde{X} \arrow[swap]{dl}{\sim} \arrow{dd} 
\\
\widetilde{Y}_0 \arrow[crossing over]{rr} \arrow{dd}
& & \widetilde{Y} \arrow[crossing over]{dd} 
\\
& X_0 \arrow{dl} \arrow{rr}
& & X \arrow{dl}
\\
Y_0 \arrow{rr}
& & Y \arrow[crossing over, leftarrow]{uu}
\\
\end{tikzcd}
\]}
Denote the normal space of \(Y_0\) in \(Y\) at \(y_0\) by \(\normalBundle[y_0]{Y_0}{Y}=T_{y_0} Y/T_{y_0} Y_0\).
Suppose that \(Y_0 \subset Y\) has codimension \(k+1\).
As is well known \cite[p. 604]{Griffiths/Harris:1978}, \(p^{-1}\of{y_0} \subset Y_0\) is canonically biholomorphic to \(\Proj{}\normalBundle[y_0]{Y_0}{Y} \cong \Proj{k}\) with
\[
\left.\normalBundle{\widetilde{Y}_0}{\widetilde{Y}}\right|_{p^{-1}\of{y_0}} \cong \OO{-1}.
\]
Therefore \(\widetilde{X}_0=\widetilde{Y}_0\) is covered in rational curves which project to points in \(Y\).

Pick any projective line, linearly embedded in projective space: \(g \colon \Proj{1} \to \Proj{k}=\Proj{}\normalBundle[y_0]{Y_0}{Y}=\widetilde{Y}_{0,y_0}\).
These lines cover the fibers of \(\widetilde{Y}_0 \to Y_0\).
Since the rational curve \(g\) projects to a point \(y_0  \in Y\), any small deformation of \(g\) as a map to \(\tilde{Y}\) maps into a Stein neighborhood of \(y_0 \in Y\), so also has image in \(Y\) a single point.
Therefore the image of any deformation of \(g\) lies inside \(\widetilde{Y}_0=\widetilde{X}_0\).
So the composition \(\bar{g}\) of 
\[
\begin{tikzcd}
\Proj{1} \arrow{r}{g} & \tilde{X} \arrow{r} & X
\end{tikzcd}
\]
has image inside \(E=X_0\).

If \(\bar{g}\) is constant for generic \(g\), then the fibers of \(\tilde{X}_0 \to X_0\) are the same as the fibers of \(\tilde{Y}_0 \to Y_0\) so \(X=Y\), a contradiction.
Suppose that \(\bar{g} \colon \Proj{1} \to X\) is free.
By freedom, there is a section \(v\) of \(\bar{g}^*TX\) not tangent to \(X_0\) at \(x_0\).
Free curves are unobstructed, so this section is the velocity of a curve in the Douady deformation space of the graph of \(\bar{g}\).
Pulling back the universal flat family to that curve makes \(\bar{g}\) a fiber in a ruled surface \(S \to X\).
In particular, the generic ruling line of \(S\) has finite intersection with \(X_0\).
The map \(S \to X\) lifts into a map \(S \to \tilde{X}\) by the universal mapping property of blowups \cite[p. 291, Cor. 2.2]{SCV:VII}.
But \(\tilde{X}=\tilde{Y}\), so this map lies in \(\tilde{Y}\).

In particular if \(\bar{g}\) lies in \(E\) then there is a rational curve in \(\widetilde{Y}_{0,y_0}\) that belongs to a ruled surface \(S\) not contained in \(\widetilde{Y}_0\), contradicting our reasoning above.
\end{proof}

\begin{corollary}\label{cor:BlowsDown.2}
Suppose that \(f \colon X \to Y\) is a modification of reduced complex spaces with \(Y\) smooth.
Suppose that there is a smooth point of \(X\), lying over a smooth point of \(Y\), near which \(f\) is not a local isomorphism.
Then the set of smooth points of \(X\) is not convex.
\end{corollary}

Corollary~\ref{cor:BlowsDown} on page~\pageref{cor:BlowsDown} follows immediately.\label{proof:cor:BlowsDown}

\begin{lemma}\label{lemma:rat.defined}
Suppose that \(M\) is a compact complex manifold containing no rational curves.
Suppose that \(X\) is a reduced and irreducible compact complex space.
Then every meromorphic map \(X \DashedArrow M\) extends to a unique holomorphic map \(X \to M\).
\end{lemma}
\begin{proof}
Let \(X' \subset X \times M\) be the graph of a meromorphic map \(X \DashedArrow M\).
Composing the inclusion and projection \(X' \to X \times M \to X\) gives a proper modification \(X' \to X\).
Assume the exceptional divisor of \(X' \to X\) is nonempty, i.e. \(X' \to X\) is not a biholomorphism.

By lemma~\ref{lemma:mod}, there is a nonfree rational curve on \(X'\) mapping to a point in \(X\) through the generic point of the exceptional divisor.
Every deformation in \(X \times M\) of this rational curve is then just precisely a deformation of the rational curve on \(M\) and a point of \(X\) to map to.
But there are no rational curves on \(M\).
Hence the exceptional divisor of \(X' \to X\) is empty and the map is defined everywhere.
\end{proof}

\begin{corollary}\label{corollary:bimero}
Any bimeromorphism of compact complex manifolds with no rational curves extends to a unique biholomorphism.
\end{corollary}

\section{Moishezon manifolds}

\begin{lemma}\label{lemma:G.P.bundle}
Every holomorphic fiber bundle with rational homogeneous fibers and smooth projective base has smooth projective total space.
\end{lemma}
\begin{proof}
The result is true for bundles with projective space fibers \cite[p. 139 Prop. 4.1]{Barth:2004}.
Take a rational homogeneous variety \(X=G/P\), where we assume \(G\) is the automorphism group of \(X\), hence a complex semisimple Lie group.
Then \(X\) admits a \(G\)-equivariant projective embedding \cite[p. 382]{Fulton/Harris:1991}: let \(P_x\) be the stabilizer of each point \(x \in X\), \(\LieP_x\) its Lie algebra, and map
\[
x \in X \mapsto \Lm{\operatorname{top}}{\LieP_x} \in \Proj{}\of{\Lm{\dim P}{\LieG}}=\Proj{r}.
\]
The total space of any holomorphic \(X\)-bundle is thus a compact embedded complex submanifold of the total space of a \(\Proj{r}\)-bundle.
The total space of the \(\Proj{r}\)-bundle is a smooth projective variety, so the total space of the \(X\)-bundle is a smooth projective variety by the Kodaira embedding theorem.
\end{proof}

We prove Corollary~\ref{corollary:Moishezon} from page~\pageref{corollary:Moishezon}.

\begin{proof}\label{proof:corollary:Moishezon}
Every Moishezon manifold \(M\) is dominated by a smooth projective variety \cite[p. 26, Thm. 3.6]{Ueno:1975} and is therefore Fujiki and so tame.
Every Moishezon manifold \(M\) which is not a smooth projective variety contains a rational curve \cite[p. 17 Thm. 3.1]{Cascini:2013}.
By theorem~\ref{thm:OneCurveForcesDescent}, the geometry on \(M\) drops to a geometry on a lower dimensional compact complex manifold \(M'\), and \(M \to M'\) is a holomorphic bundle of rational homogeneous varieties, and \(M'\) contains no rational curves.
The image \(M'\) of any surjective morphism of complex manifolds \(M \to M'\) from a Moishezon manifold \(M\)  is Moishezon \cite[p. 27, Cor. 3.9]{Ueno:1975}.
But \(M'\) contains no rational curves, so \(M'\) is a smooth projective variety \cite[p. 26, Thm. 3.6]{Ueno:1975} and \(M \to M'\) is a holomorphic fiber bundle with rational homogeneous fibers.
By lemma~\ref{lemma:G.P.bundle}, \(M\) is a smooth projective variety and \(M \to M'\) is a regular fibration of smooth projective varieties.
\end{proof}

We prove Corollary~\vref{corollary:ample}:
\begin{proof}\label{proof:corollary:ample}
Every rational homogeneous variety contains a rational curve with ample ambient tangent bundle \cite[p. 20 Lemma 16]{McKay:2006}, therefore (2) implies (1).

Suppose (1).
The Cartan geometry is locally isomorphic to its model by Corollary~\vref{corollary:ample.rational}.
Verbitsky \cite{Verbitsky:2014} proves that if a complex manifold contains a rational curve with ample normal bundle then it admits an open embedding to a Moishezon manifold.
In our case, since \(M\) is a compact complex manifold, \(M\) is Moishezon.
The manifold \(M\) is a smooth projective variety by Corollary~\vref{corollary:Moishezon}.
Moreover the Cartan geometry on \(M\) drops by Theorem~\vref{thm:OneCurveForcesDescent}, and the rational curve lies in the fiber of the maximal rationally connected fibration.
The ample ambient tangent bundle prevents the rational curve living in a fiber of a fibration unless the fibers are points or the base is a point, since the deformations of the rational curve have velocities spanning the ambient tangent bundle along the curve.
Therefore the model is some compact complex homogeneous space and \(M\) is isomorphic to it.
The model is a rational homogeneous variety by Proposition~\vref{proposition:RationallyConnectedHomogeneous}.
Indeed by the classification of homogeneous projective varieties \cite[p. 216 Thm 4.2]{Khenkin:1990}, there are no rational curves with ample ambient tangent bundle in any homogeneous projective variety \(X\) unless \(X\) is a rational homogeneous variety.
\end{proof}

\section{Meromorphic and holomorphic symmetries}%
\label{section:symmetries}

We prove theorem~\ref{theorem:bimero} from page~\pageref{theorem:bimero}.
\begin{proof}\label{proof:theorem:bimero}
If neither \(M_0\) nor \(M_1\) contain rational curves, the result follows from corollary~\vref{corollary:bimero}. 

By theorem~\vref{thm:FreeRats}, both \(M_0\) and \(M_1\) are convex.
Suppose that \(M_0\) contains a rational curve and that \(M_1\) does not.
By lemma~\vref{lemma:rat.defined}, \(M_0 \DashedArrow M_1\) extends to a unique holomorphic map, a local biholomorphism away from a closed complex hypersurface.
By convexity, \(M_0\) contains a rational curve through the generic point, and in particular contains a rational curve not sitting inside that hypersurface.
Therefore \(M_1\) contains a rational curve, a contradiction.
Vice versa: \(M_0\) contains a rational curve just when \(M_1\) does.

If there are rational curves in both of \(M_0\) and \(M_1\), then by theorem~\ref{thm:OneCurveForcesDescent} we can drop both geometries:
{\footnotesize
\[
\begin{tikzcd}
M_0 \arrow[<->,dashed]{r} \arrow{d} & M_1 \arrow{d} \\
M'_0 & M_1'
\end{tikzcd}
\]}
Note that we have yet to determine whether the Cartan geometries on \(M_0'\) and \(M_1'\) have the same model.
The composed meromorphic maps
{\footnotesize
\[
\begin{tikzcd}
M_0 \arrow[dashed]{dr} \arrow[<->,dashed]{r} \arrow{d} & M_1 \arrow{d} \arrow[dashed]{dl}\\
M'_0 & M_1'
\end{tikzcd}
\]}
are holomorphic
{\footnotesize
\[
\begin{tikzcd}
M_0 \arrow{dr} \arrow[<->,dashed]{r} \arrow{d} & M_1 \arrow{d} \arrow{dl}\\
M'_0 & M_1'
\end{tikzcd}
\]}
\par\noindent{}by lemma~\ref{lemma:rat.defined}.
The generic point of \(M_0\) lies on a rational curve.
Each rational curve in \(M_0\) lies in a fiber of \(M_0 \to M_0'\), and these fibers are rationally connected.
The generic rational curve in \(M_0\) maps meromorphically to \(M_1'\), so the fibers of \(M_0 \to M_1'\) contain the fibers of \(M_0 \to M_0'\).
Therefore these holomorphic maps descend:
{\footnotesize
\[
\begin{tikzcd}
M_0 \arrow{dr} \arrow[<->,dashed]{r} \arrow{d} & M_1 \arrow{d} \arrow{dl}\\
M'_0 \arrow[<->]{r} & M_1'
\end{tikzcd}
\]}
The base manifolds are biholomorphic and we can thereby identify them:
{\footnotesize
\[
\begin{tikzcd}
M_0 \arrow{dr} \arrow[<->,dashed]{rr} & & M_1 \arrow{dl}\\
& M' & 
\end{tikzcd}
\]}

Careful: for example, if we take \(M_0=\Proj{2}\) and \(M_1=\Proj{1} \times \Proj{1}\), i.e. different models of holomorphic Cartan geometries, these are bimeromorphic, but not biholomorphic, and \(M_0'=M_1'=\Set{*}\) are biholomorphic.

But we are now assuming that the bimeromorphism \(M_0 \DashedArrow[<->] M_1\) is an isomorphism of the Cartan geometries, on dense open sets.
This means precisely that the bimeromorphism \(M_0 \DashedArrow[<->] M_1\) is induced by an \(H\)-equivariant bimeromorphism \(E_0 \DashedArrow[<->] E_1\) identifying the Cartan connections.
Moreover \(E_0 \DashedArrow[<->] E_1\) is \(H'\)-equivariant, because both Cartan geometries drop to the same model.
This bimeromorphism \(E_0 \DashedArrow[<->] E_1\) is invariant under the flows of the vector fields \(\vec{A}\), which have all of \(E_0\) and \(E_1\) as orbits.
Therefore the bimeromorphism extends to \(E_0 \to E_1\) uniquely to an \(H'\)-equivariant biholomorphism.
\end{proof}

\section{Vector fields}

Take a \(C^{\infty}\) Cartan geometry \(\pi \colon E \to M\).
A \(C^{\infty}\) \emph{infinitesimal symmetry} \(v\) of the geometry is a \(C^{\infty}\) vector field on \(E\), commuting with the right \(H\)-action on the bundle \(E \to M\), so that the flow of \(v\) preserves the Cartan connection.
Commutation with the \(H\)-action ensures that there is a unique vector field \(\bar{v}\) on \(M\) so that \(\pi_* v = \bar{v}\).

If \(E \to M\) is a \(C^{\infty}\) Cartan geometry and \(f \colon M_0 \to M\) is a local diffeomorphism, then \(E_0 \defeq f^* E \to M_0\) is the total space of a unique Cartan geometry, the \emph{pullback geometry}, with the same model, with Cartan connection pulled back via the obvious map \(E_0 \to E\).

Pullback the standard flat projective connection on \(\Proj{1}\) to a disk inside an affine chart.
The projective connection is invariant under projective automorphisms, so under all biholomorphisms of the disk.
Therefore the projective connection descends to all connected Riemann surfaces of genus \(\ge 2\).
The local infinitesimal symmetries of the projective connection are the same on these Riemann surfaces as they are on the projective line.
But the global infinitesimal symmetries are trivial on any compact higher genus Riemann surface, because the biholomorphism group of the Riemann surface is finite and the only holomorphic vector field is zero.

\begin{lemma}[Amores \cite{Amores:1979} p. 5 Thm. 3.3]\label{lemma:local.sym}
Take a connected manifold \(M\) with a real analytic \(V,H\)-geometry \(E \to M\) and a connected open set \(U \subset M\).
Restrict the Cartan geometry to \(U\).
Every infinitesimal symmetry of the restricted Cartan geometry pulls back to some covering space of \(M\) and then extends uniquely to an infinitesimal symmetry on that covering space.
\end{lemma}
The idea behind the proof is easy: every infinitesimal symmetry \(v\) satisfies 
\[
0=\left[v,\vec{A}\right]
\]
for every \(A \in \LieG\).
Therefore \(v\) is invariant under the flows of all of these \(\vec{A}\) vector fields.
But these vector fields span the tangent spaces of \(E\), so act locally transitively, extending \(v\) to be defined everywhere on some covering space of \(E\) (to handle potential monodromy), preserving \(\omega\) and commuting with the \(\vec{A}\) by analytic continuation.

\begin{corollary}
On a complex manifold \(M\), every meromorphic infinitesimal symmetry of a holomorphic Cartan geometry is holomorphic.
\end{corollary}
\begin{proof}
The infinitesimal symmetry commutes with the vector fields \(\vec{A}\), so extends to be invariant along their flows, giving a global holomorphic extension.
\end{proof}

\section{Relation to the literature}

Biswas and Bruzzo \cite{Biswas/Bruzzo:2011} proved a special case of the main theorem above.
Suppose that \(M\) is a compact connected K\"ahler manifold with numerically effective tangent bundle. 
Then some finite unramified covering of \(M\) admits a holomorphic surjective submersion to a complex torus \cite{Demailly/Peternell/Schneider:1994} with Fano fibers (and hence rationally connected fibers). 
Biswas and Bruzzo prove that if \(M\) admits a holomorphic Cartan geometry, then these fibers are rational homogeneous varieties, which is a consequence of the main theorem above.

The authors \cite{Biswas/McKay:2009} also previously proved a special case of the main theorem above: rationally connected compact complex manifolds carrying a holomorphic Cartan geometry are rational homogeneous varieties. 
We also proved in that paper that any compact K\"ahler manifold with trivial canonical bundle bearing a holomorphic Cartan geometry admits a finite unramified covering by a complex torus.
Dumitrescu \cite{Dumitrescu:2009} proved the same result for structures of affine type,
and we previously proved the same result for smooth projective varieties \cite{Biswas/McKay:2008}.
If the holomorphic Cartan geometry is parabolic, then it pulls back to a translation
invariant holomorphic Cartan geometry on the torus \cite{McKay:2008c}.
Dumitrescu and McKay conjecture that all flat holomorphic Cartan geometries on any complex torus are translation invariant.

The explicit classification of holomorphic Cartan geometries on uniruled compact complex surfaces \cite{McKay:2011} employs the main theorem above. 
Say that a complex homogeneous space \((X,G)\) is \emph{rigid} among smooth projective varieties (or Moishezon manifolds) if any holomorphic \((X,G)\)-geometry with that model can only exist on a unique smooth projective variety (or Moishezon manifold): the model \(X\), and the complex manifold \(X\) admits a unique holomorphic \((X,G)\)-geometry. 
The main theorem also ensures that certain complex homogeneous spaces \((X,G)\) are rigid among smooth projective varieties \cite{McKay:2011b}.
Corollary~\vref{corollary:Moishezon} extends this rigidity to rigidity among Moishezon manifolds.
A weaker form of rigidity occurs for a larger class of holomorphic Cartan geometries  \cite{McKay:2007}: the relations among characteristic classes on the model are also satisfied on any compact K\"ahler manifold with such a geometry.

\section*{Acknowledgements.}
We are very grateful for the assistance of Najmuddin Fakhruddin, Amit Hogadi, Sergei Ivashkovich, J\'anos Koll\'ar, Andrei Musta\c{t}\u{a} and an anonymous referee.
This publication has emanated from activity conducted with the financial support of Science Foundation Ireland under the International Strategic Cooperation Award Grant Number SFI/13/ISCA/2844.

\end{document}